\documentclass[reqno,11pt,twoside]{article}
\usepackage[hmargin=1.25in,vmargin=1.25in, a4paper, centering]{geometry}
\usepackage{fourier}
\usepackage{ebgaramond}
\usepackage{amsmath, amsthm, amssymb, bm}
\usepackage{graphicx}
\usepackage{xcolor}
\usepackage{enumerate}
\usepackage[format=hang, margin=20pt, font=small, labelfont=bf, labelsep=colon]{caption}
\usepackage{booktabs,array}
\usepackage[section]{placeins}
\usepackage{etoolbox}
\numberwithin{equation}{section}
\theoremstyle{plain}
\newtheorem{thm}{Theorem}[section]
\newtheorem{fact}[thm]{Fact}
\newtheorem{lem}[thm]{Lemma}
\newtheorem{assumption}[thm]{Assumption}
\theoremstyle{definition}

\theoremstyle{remark}
\newtheorem{rem}[thm]{Remark}
 \newcommand{\addQEDstyle}[2]{\AtBeginEnvironment{#1}{\pushQED{\qed}\renewcommand{\qedsymbol}{#2}}\AtEndEnvironment{#1}{\popQED}}
 \addQEDstyle{ex}{$\triangleleft$}
  \addQEDstyle{defn}{$\triangleleft$}
 \addQEDstyle{rem}{$\blacktriangleleft$}
\usepackage[noadjust]{marginnote}

\usepackage[toc,titletoc]{appendix}
\appendixtocoff
\usepackage{etoolbox, apptools}
\AtBeginEnvironment{appendices}{\appendixtrue}
\usepackage[nobottomtitles,pagestyles]{titlesec}
\usepackage{emptypage}
\widenhead*{0.5in}{0.5in}
\renewpagestyle{plain}[\bfseries]{\footrule\setfoot{}{\thepage}{}}
\newpagestyle{mystyle}[\bfseries]{
\headrule
\sethead[\thepage][][M. Capoferri, L. Friedlander, M. Levitin, and D. Vassiliev]{Two-term spectral asymptotics in elasticity}{}{\thepage}
}
\renewcommand\footnotemark{}
\titleformat{\section}
{\normalfont\large\bfseries}
{\IfAppendix{\filcenter\appendixname~\thesection.}{\filcenter\S\thesection.}}{1ex}{\filcenter}
\titleformat{\subsection}[runin]
{\normalfont\normalsize\bfseries}{\S\thesubsection.}{1ex}{}[.]
\newcommand{\Ric}{\operatorname{Ric}}
\newcommand{\Vol}{\operatorname{Vol}}
\newcommand{\grad}{\operatorname{grad}}
\newcommand{\Div}{\operatorname{div}}
\newcommand{\curl}{\operatorname{curl}}
\renewcommand{\tilde}{\widetilde}
\newcommand{\ir}{\mathrm{i}}
\newcommand{\dr}{\mathrm{d}}
\newcommand{\er}{\mathrm{e}}

\newcommand{\Dir}{\mathrm{Dir}}
\newcommand{\free}{\mathrm{free}}
\newcommand{\Liu}{\mathrm{Liu}}
\newcommand{\plane}{\mathrm{plane}}
\usepackage[colorlinks,allcolors=blue,pagebackref]{hyperref}
\usepackage{pifont}
\renewcommand*{\backrefalt}[4]{%
\ifcase #1 %
No citations%
\or
\ding{43}~p.~#2%
\else
\ding{43}~pp.~#2%
\fi}
\begin{document}

\title{Two-term spectral asymptotics in linear elasticity%
\footnote{This version corrects typesetting misprints in formula \eqref{eq:seck} in the previous arXiv and published versions. These misprints do not affect the rest of the paper. We are grateful to A. Henrot for pointing these misprints out to us.}
\footnote{For the sake of transparency, authors' computational scripts are available at \url{https://michaellevitin.net/elasticity.html}.}
\footnote{Authors' typesetting may differ from that in the journal version.}
\\[1ex]
\small Journal of Geometric Analysis \textbf{33} (2023), article 242. DOI: \href{https://doi.org/10.1007/s12220-023-01269-y}{10.1007/s12220-023-01269-y}. arXiv: \href{http://arxiv.org/abs/2207.13636}{2207.13636v3}. 
}
\author{
Matteo Capoferri\thanks{\textbf{MC}:
Maxwell Institute for Mathematical Sciences and
Department of Mathematics, Heriot-Watt University, Edinburgh EH14 4AS, UK; m.capoferri@hw.ac.uk,
\url{https://mcapoferri.com}.
}
\and
Leonid Friedlander\thanks{\textbf{LF}:
Department of Mathematics, The University of Arizona, 617 N.~Santa Rita Ave.~P.O.~Box 210089 Tucson, AZ 85721-0089 USA;
friedlan@math.arizona.edu,
\url{https://www.math.arizona.edu/people/lfriedla}.
}
\and
Michael Levitin\thanks{\textbf{ML}:
Department of Mathematics and Statistics,
University of Reading,
Pepper Lane,
Whiteknights,
Reading RG6~6AX,
UK;
M.Levitin@reading.ac.uk,
\url{https://www.michaellevitin.net}.
}
\and
Dmitri Vassiliev\thanks{\textbf{DV}:
Department of Mathematics,
University College London,
Gower Street,
London WC1E~6BT,
UK;
D.Vassiliev@ucl.ac.uk,
\url{http://www.ucl.ac.uk/\~ucahdva/}.
}}

\renewcommand\footnotemark{}
\date{12 April 2024}

\maketitle
\begin{abstract}
We establish the two-term spectral asymptotics for boundary value problems of linear elasticity on a smooth compact Riemannian manifold of arbitrary dimension. We also present some illustrative examples and give a historical overview of the subject. In particular, we correct erroneous results published in [J.~Geom.~Anal. \textbf{31} (2021), 10164--10193].
\end{abstract}

{\small \tableofcontents}

\section{Statement of the problem and main results}\label{sec:mainresults}

The aim of this paper is to find explicitly the second asymptotic term for the eigenvalue counting functions of the operator of linear elasticity on a smooth $d$-dimensional Riemannian manifold with boundary, equipped with either Dirichlet or free boundary conditions. The main body of the paper is devoted to the proof of our main result,  stated later in this section as Theorem~\ref{thm:maintheorem}, using the strategy based on an algorithm due to Vassiliev \cite{Vas84, SaVa}.

Our paper is in part motivated by incorrect statements published in \cite{Liu}, on two-term asymptotic expansions for the heat kernel of the operator of linear elasticity in the same setting.  A discussion of  \cite{Liu} is presented in Remark~\ref{rem:Liu} and continues further in Appendix~\ref{appendix:Liu}.  In two further Appendices \ref{sec:ex2d} and  \ref{sec:ex3d} we provide an ``experimental" verification of the correctness of our results, by numerically computing the quantities in question for explicit examples in dimensions two and three.

\

Let $(\Omega, g)$ be a smooth compact connected
$d$-dimensional ($d\ge2$) Riemannian manifold with boundary $\partial\Omega\ne\emptyset$.  We consider the linear elasticity operator $\mathcal{L}$ acting on vector fields $\mathbf{u}$ and defined by\footnote{We use standard tensor notation and employ Einstein's summation convention throughout.}
\begin{equation}\label{eq:elasticity}
 (\mathcal{L}\mathbf{u})^\alpha:=-\mu\left(\nabla_\beta\nabla^\beta u^\alpha+{\Ric^\alpha}_\beta u^\beta\right)-(\lambda+\mu)\nabla^\alpha\nabla_\beta u^\beta.
\end{equation}
Here and further on $\nabla$ is the Levi--Civita connection associated with $g$, 
$\Ric$ is Ricci curvature, and $\lambda$, $\mu$ are real constants called \emph{Lam\'e coefficients} which are assumed to satisfy\footnote{Without loss of generality, the second condition in \eqref{eq:mulambda} may be replaced by a physically meaningless but less restrictive condition $\lambda+\mu>0$, in which case \eqref{eq:alphabound} becomes $\alpha\in(0,1)$.}
\begin{equation}\label{eq:mulambda}
\mu>0, \qquad d\lambda+2\mu>0.
\end{equation}
We will also use the parameter
\begin{equation}
\label{eq:alpha}
\alpha:=\frac{\mu}{\lambda+2\mu}.
\end{equation}
Subject to \eqref{eq:mulambda}, we have
\begin{equation}
\label{eq:alphabound}
\alpha\in\left(0, \frac{d}{2(d-1)}\right)\subseteq(0,1). 
\end{equation}
We also assume that the material density of the elastic medium, $\rho_\mathrm{mat}\,$, is constant. More precisely, we assume that $\rho_\mathrm{mat}$ differs from the Riemannian density $\sqrt{\det g}\ $ by a constant positive factor.

We complement \eqref{eq:elasticity} with suitable boundary conditions, for example the \emph{Dirichlet} condition 
\begin{equation}\label{eq:Dir}
\mathbf{u}|_{\partial\Omega}=0
\end{equation}
(sometimes called the \emph{clamped edge} condition in the physics literature), or the \emph{free boundary} condition
\begin{equation}\label{eq:trac}
\mathcal{T}\mathbf{u}|_{\partial\Omega}=0
\end{equation}
(sometimes called the \emph{free edge} or \emph{zero traction} condition in the physics literature,
and also the \emph{Neumann} condition\footnote{We will not call \eqref{eq:trac} the Neumann condition in order to avoid confusion with the erroneous ``Neumann'' condition in \cite{Liu}.}),
where $\mathcal{T}$ is the boundary traction operator defined by
\begin{equation}\label{eq:T}
(\mathcal{T}\mathbf{u})^\alpha:=\lambda n^\alpha \nabla_\beta u^\beta  +\mu\left(n^\beta \nabla_\beta u^\alpha + n_\beta \nabla^\alpha u^\beta\right).
\end{equation}
Here $\mathbf{n}$ is the exterior unit normal vector to the boundary $\partial\Omega$.\footnote{One can also consider \emph{mixed} boundary value problems by imposing zero traction conditions only in some of the $d$ directions, and Dirichlet conditions in the remaining directions; for an example of such a problem motivated by applications see e.g.~\cite{LMS}.} 

It is easy to verify that  subject to the restrictions \eqref{eq:mulambda} the operator $\mathcal{L}$ is elliptic.  Its principal symbol has eigenvalues 
\[
(\lambda+2\mu)\|\xi\|^2\quad\text{(with multiplicity one)},\qquad  \mu\|\xi\|^2\quad\text{(with multiplicity $d-1$)}.
\]
Here and further on $\|\xi\|$ denotes the Riemannian norm of the covector $\xi$.
The quantities $\sqrt{\lambda+2\mu}$ and $\sqrt{\mu}$ are known as the \emph{speeds of propagation of longitudinal and transverse
elastic waves}, respectively.

It is also easy to verify that either of the boundary conditions \eqref{eq:Dir} and \eqref{eq:trac} is of the Shapiro--Lopatinski type \cite{KruTuo} for $\mathcal{L}$, and therefore the corresponding boundary value problems are elliptic\footnote{We remark that the ellipticity of the corresponding boundary value problems may be broken if $\lambda$, $\mu$ lie outside the  range $\mu>0$, $\lambda+\mu>0$. This is the subject of a distinct but very interesting \emph{Cosserat problem}, see e.g.~\cite{SimvWa} or \cite{Levcosserat}.}. 

The boundary conditions \eqref{eq:Dir} and \eqref{eq:trac} are linked by \emph{Green's formula} for the elasticity operator, 
\begin{equation}\label{eq:green} 
\left(\mathcal{L}\mathbf{u}, \mathbf{u}\right)_{L^2(\Omega)}=\mathcal{E}[\mathbf{u}]-\left(\mathcal{T}\mathbf{u}, \mathbf{u}\right)_{L^2(\partial\Omega)}\,,
\end{equation}
where the quadratic form
\begin{equation}\label{eq:E}
\mathcal{E}[\mathbf{u}] := \int_\Omega \left(\lambda\left(\nabla_\alpha u^\alpha\right)^2+\mu\left(\nabla_\alpha u_\beta+\nabla_\beta u_\alpha\right)\nabla^\alpha u^\beta\right)\, \sqrt{\det g}\ \dr x
\end{equation}
equals twice the \emph{potential energy} of elastic deformations associated  with displacements $\mathbf{u}$, and is non-negative for all $\mathbf{u}\in H^1(\Omega)$ and strictly positive for all $\mathbf{u}\in H^1_0(\Omega)$.
The structure of the quadratic functional \eqref{eq:E} of linear elasticity is the result of certain geometric assumptions, see  \cite[formula (8.28)]{part1}, as well as \cite[Example 2.3 and formulae (2.5a), (2.5b) and (4.10e)]{diffeo}.

Consider the Dirichlet eigenvalue problem 
\begin{equation}\label{eq:Lspec}
\mathcal{L}\mathbf{u}=\Lambda \mathbf{u},
\end{equation}
subject to the boundary condition \eqref{eq:Dir}, where $\Lambda$ denotes the spectral parameter.
The spectral parameter $\Lambda$ has the physical meaning $\,\Lambda=(\rho_\mathrm{mat}/\sqrt{\det g}\,)\,\omega^2$, where $\rho_\mathrm{mat}$ is the material density, $\sqrt{\det g}\ $ is the Riemannian density and $\omega$ is the angular natural frequency of oscillations of the elastic medium.
With account of ellipticity and Green's formula \eqref{eq:green}, it is a standard exercise to show that one can associate with \eqref{eq:Lspec}, \eqref{eq:Dir}, the spectral problem for a self-adjoint elliptic operator $\mathcal{L}_\Dir$ with form domain $H^1_0(\Omega)$; we omit the details. The spectrum of the problem is discrete and consists of isolated eigenvalues 
\begin{equation}\label{eq:Direv}
(0<)\Lambda_1^\Dir\le \Lambda_2^\Dir\le\dots
\end{equation}
enumerated with account of multiplicities and accumulating to $+\infty$. A similar statement holds for the free edge boundary problem \eqref{eq:Lspec}, \eqref{eq:trac}, which is associated with a self-adjoint operator $\mathcal{L}_\free$ whose form domain is   $H^1(\Omega)$; we denote its eigenvalues by
\[
(0\le)\Lambda_1^\free\le \Lambda_2^\free\le\dots.
\]

We associate with the spectrum \eqref{eq:Direv} of the Dirichlet elasticity problem on $\Omega$ the following functions. Firstly, we introduce the \emph{eigenvalue counting function} 
\begin{equation}\label{eq:NDir}
\mathcal{N}_\Dir(\Lambda):=\#\left\{n: \Lambda_n^\Dir<\Lambda\right\},
\end{equation}
defined for $\Lambda\in\mathbb{R}$. Obviously, $\mathcal{N}_\Dir(\Lambda)$ is monotone increasing in $\Lambda$, takes integer values, and is identically zero for $\Lambda\le\Lambda_1^\Dir$. An analogous eigenvalue counting function of the free boundary problem will be denoted $\mathcal{N}_\free(\Lambda)$\footnote{In what follows, we will write $\mathcal{N}(\Lambda)$ if a corresponding statement is true for either $\mathcal{N}_\Dir(\Lambda)$ or $\mathcal{N}_\free(\Lambda)$ irrespective of the boundary conditions.}.

Secondly, we introduce the \emph{partition function}, or the \emph{trace of the heat semigroup},
\begin{equation}\label{eq:ZDir}
\mathcal{Z}_\Dir(t):=\operatorname{Tr}\,\mathrm{e}^{-t\mathcal{L}_\Dir}=\sum_{m=1}^\infty \mathrm{e}^{-t\Lambda_m^\Dir}\,,
\end{equation}
defined for $t>0$ and monotone decreasing in $t$. The free boundary partition function $\mathcal{Z}_\free(t)$ is defined in the same manner\footnote{In what follows, we will write $\mathcal{Z}(t)$ if a corresponding statement is true for either $\mathcal{Z}_\Dir(t)$ or $\mathcal{Z}_\free(t)$ irrespective of the boundary conditions.}.

The existence of asymptotic expansions of $\mathcal{N}(\Lambda)$ as $\lambda\to+\infty$ and of $\mathcal{Z}(t)$ as $t\to 0^+$, and precise expressions for the coefficients of these expansions in terms of  the geometric invariants of $\Omega$, for either the Dirichlet or the free boundary conditions, and similar questions for the Dirichlet and Neumann Laplacians, have been a topic of immense interest among mathematicians and physicists since the publication of the first edition of Lord Rayleigh's\footnote{The 1904 Nobel Laureate (Physics).} \emph{The Theory of Sound} in 1877 \cite{Ra77}. A detailed historical review of the  field is beyond the scope of this article; we refer the interested reader to \cite{SaVa}, \cite{ANPS}, and \cite{Ivr}, and references therein. 

Before stating our main results, we summarise below some known facts concerning the asymptotics of  \eqref{eq:NDir} and \eqref{eq:ZDir}, and their free boundary analogues.  Further on, we always assume that $(\Omega,g)$ is a $d$-dimensional Riemannian manifold satisfying the conditions stated at the beginning of the article.

\begin{fact}\label{fact:1}
For any $(\Omega, g)$ we have
\begin{equation}\label{eq:N1term}
\mathcal{N}(\Lambda)=a\,\Vol_d(\Omega)\,\Lambda^{d/2} + o\left(\Lambda^{d/2}\right)\quad\text{as}\quad\Lambda\to+\infty,
\end{equation} 
where
\begin{equation}\label{eq:tildeC} 
a = \frac{1}{(4\pi)^{d/2}\Gamma\left(1+\frac d2\right)} \left(\frac{d-1}{\mu^{d/2}}+\frac{1}{(\lambda+2\mu)^{d/2}}\right)
\end{equation}  
is the \emph{Weyl constant} for linear elasticity, and $\Vol_d(\Omega)$ denotes the Riemannian volume of $\Omega$.
\end{fact}

This immediately implies 
\begin{fact}\label{fact:2}
For any $(\Omega, g)$ we have
\begin{equation}\label{eq:Z1term}
\mathcal{Z}(t)=\widetilde{a}\,\Vol_d(\Omega)\,t^{-d/2} +o\left(t^{-d/2}\right)\qquad\text{as }t\to 0^+,
\end{equation} 
with 
\begin{equation}\label{eq:C}
\widetilde{a}=\Gamma\left(1+\frac{d}{2}\right)\,a\,. 
\end{equation}
\end{fact}

The one-term asymptotic law \eqref{eq:Z1term}, \eqref{eq:C} was established, at a physical level of rigour, by P. Debye\footnote{The 1936 Nobel Laureate (Chemistry).} \cite{Debye} in 1912\footnote{Debye (and many other physicists following him) was in fact studying not the partition function but a closely related quantity called the \emph{specific heat} of $\Omega$. The asymptotic behaviour of these two quantities follow from each other; we omit the details here and further on in order not to overload this paper with physical background.\label{footnote:physics}}. The one-term asymptotics   \eqref{eq:N1term}, \eqref{eq:tildeC} was rigorously proved\footnote{Strictly speaking, for $d=3$ in the Euclidean case only, but the generalisation is trivial in view of subsequent advances.} by H. Weyl in 1915 \cite{We15}.  We note that \eqref{eq:Z1term}, \eqref{eq:C} immediately follow from \eqref{eq:N1term}, \eqref{eq:tildeC}  since the partition function $\mathcal{Z}(t)$ is just the Laplace transform of the (distributional) derivative of the counting function $\mathcal{N}(\Lambda)$,
\begin{equation}\label{eq:LT}
\mathcal{Z}(t) =\int_{-\infty}^\infty \mathrm{e}^{-t \Lambda} \mathcal{N}'(\Lambda) \,\mathrm{d}\Lambda\,.
\end{equation} 

We also have, see e.g. \cite{Grubb} and also \cite[Remark 4.1(ii)]{Liu}, the following
\begin{fact}\label{fact:3}
Let  $\aleph\in\{\Dir, \free\}$. Then
\begin{equation}\label{eq:Z2term}
\mathcal{Z}_\aleph(t)=\widetilde{a} \,\Vol_d(\Omega) t^{-d/2} + \widetilde{b}_\aleph \Vol_{d-1}(\partial\Omega) t^{-(d-1)/2} + o\left(t^{-(d-1)/2}\right)\quad\text{as}\quad t\to 0^+,
\end{equation} 
with some constants $\widetilde{b}_\aleph$.  The quantity $\Vol_{d-1}(\partial\Omega)$ is the volume of the boundary $\partial \Omega$ as a $(d-1)$-dimensional Riemannian manifold with metric induced by $g$.
\end{fact}


We note that the expansions \eqref{eq:Z2term} do not in themselves imply the existence of two-term asymptotic formulae for $\mathcal{N}_\aleph(\Lambda)$. However, formula \eqref{eq:LT} implies the following
\begin{fact}\label{fact:4}
Let  $\aleph\in\{\Dir, \free\}$, and suppose that we have 
\begin{equation}\label{eq:N2term}
\mathcal{N}_\aleph(\Lambda)=a\, \Vol_d(\Omega)\Lambda^{d/2} + b_\aleph \Vol_{d-1}(\partial{\Omega}) \Lambda^{(d-1)/2} + o\left(\Lambda^{(d-1)/2}\right)\quad\text{as}\quad\Lambda\to+\infty,
\end{equation} 
with some constant $b_\aleph$. Then  
\begin{equation}\label{eq:BvstildeB}
\widetilde{b}_\aleph=\Gamma\left(1+\frac{d-1}{2}\right) b_\aleph. 
\end{equation}
\end{fact}

In general, the validity of two-term asymptotic expansions \eqref{eq:N2term} is still an open question (as it is for the scalar Dirichlet or Neumann Laplacian). However, similarly to the scalar case, there exist sufficient conditions which guarantee that   \eqref{eq:N2term} hold. These conditions are expressed in terms of the corresponding \emph{branching Hamiltonian billiards} on the cotangent bundle $T^*\Omega$, see \cite{Vas86} for precise statements.
\begin{fact}\label{fact:5}
Suppose that $(\Omega, g)$ is such that the corresponding billiards is neither \emph{dead-end} nor \emph{absolutely periodic}. Then \eqref{eq:N2term} holds for both the Dirichlet and the free boundary conditions.
\end{fact}

Fact \ref{fact:5} is a re-statement of  a more general
\cite[Theorem~6.1]{Vas84} which is applicable to the elasticity operator $\mathcal{L}$ since the multiplicities of the eigenvalues of its principal symbol are constant  on $T^*\Omega$, as we have mentioned previously.

We conclude this overview with the following observation, see also \cite[Remark 4.1(i)]{Liu}.
\begin{fact}\label{fact:6} 
The coefficients  $\widetilde b_\aleph$
are numerical constants which do not contain any information on the geometry of the manifold $\Omega$ or its boundary $\partial\Omega$.
Therefore, to determine these coefficients it is enough to find them in the Euclidean case. 
\end{fact}

Fact \ref{fact:6} easily follows from a rescaling argument: stretch $\Omega$ by a linear factor $\kappa>0$, note that the eigenvalues  then rescale as $\kappa^{-2}$, and check the rescaling of the geometric invariants and of  \eqref{eq:Z2term}.

Fact  \ref{fact:6} allows us to work from now on in the Euclidean setting, in which case \eqref{eq:elasticity} simplifies to 
\begin{equation}\label{eq:elasticityE}
\mathcal{L}\mathbf{u}=\left(-\mu\boldsymbol{\Delta}-(\lambda+\mu)\grad\Div\right)\mathbf{u}\,,
\end{equation}
where the vector Laplacian $\boldsymbol{\Delta}$ is a diagonal $d\times d$ operator-matrix having a scalar Laplacian $\Delta:=\sum_{k=1}^d \frac{\partial^2}{\partial x_k^2}$ in each diagonal entry.
In dimensions $d=2$ and $d=3$,  \eqref{eq:elasticityE} simplifies further to
\begin{equation*}
\mathcal{L}\mathbf{u}=\left(\mu\curl\curl-(\lambda+2\mu)\grad\Div\right)\mathbf{u}\,.
\end{equation*}
Note that we define the curl of a planar vector field by embedding $\mathbb{R}^2$ into $\mathbb{R}^3$; thus, $\curl\curl$ applied to a planar vector field is a planar vector field.

We are now in a position to state the main results of this paper.  Before doing so, let us introduce some additional notation.

Let
\begin{equation}
\label{eq:R(w)}
R_\alpha(w):=w^3-8w^2+8\left(3-2\alpha\right)w+16\left(\alpha-1\right)\,.
\end{equation}
The cubic equation $R_\alpha(w)=0$ has three roots $w_j$, $j=1,2,3$,  over $\mathbb{C}$, where $w_1$ is the distinguished real root in the interval $(0,1)$. We further define
\begin{equation}
\label{eq:gammaR}
\gamma_R:=\sqrt{w_1}\,.
\end{equation}
\begin{rem}
\label{rem:Rayleigh}
The subscript $R$ in $\gamma_R$ stands for ``Rayleigh". Indeed, the quantity 
\[
c_R:=\sqrt{\mu}\, \gamma_R
\]
has the physical meaning of velocity of the celebrated Rayleigh's surface wave \cite{Ra77,Ra85}. The cubic equation 
\begin{equation}
\label{eq:Rayleigh}
R_\alpha(w)=0
\end{equation}
is often referred to as \emph{Rayleigh's equation},  and it admits the equivalent formulation
\begin{equation}
\label{eq:RayleighEquiv}
\tilde R_\alpha(w):=4\sqrt{(1-\alpha w)\left(1-w\right)}-(w-2)^2=0\,.
\end{equation}
Equation \eqref{eq:Rayleigh} can be obtained from \eqref{eq:RayleighEquiv} by multiplying through by $4\sqrt{(1-\alpha w)\left(1-w\right)}+(w-2)^2$ and dropping the common factor $w$ corresponding to the spurious solution $w=0$. It is well-known \cite{RaBa,ViOg} that for all $\alpha\in (0,1)$ equation \eqref{eq:Rayleigh} --- or, equivalently, \eqref{eq:RayleighEquiv} --- admits precisely one real root $w_1=\gamma_R^2\in(0,1)$.  The nature of the other two roots $w_j$, $j=2,3$, depends on $\alpha$; we will revisit this in \S\ref{subsec:freeodd}.

Observe that $\gamma_R$ can be equivalently defined as the unique real root in $(0,1)$ of the sextic equation $R_\alpha(\gamma^2)=0$, see also \cite[\S 6.3]{SaVa}.

As we shall see,  the Rayleigh wave contributes to the second asymptotic term in the free boundary case.
\end{rem}

\begin{thm}
\label{thm:maintheorem}
Let $(\Omega, g)$ be a smooth compact connected
$d$-dimensional Riemannian manifold with boundary $\partial\Omega$. Then the second asymptotic coefficients in the two-term expansion \eqref{eq:N2term} for the eigenvalue counting function of the elasticity operator \eqref{eq:elasticity} with Dirichlet and free boundary conditions read
\begin{equation}
\begin{split}
\label{eq:maintheoremDir}
b_\Dir=-\frac{\mu^{\frac{1-d}{2}}}{2^{d+1}\pi^{\frac{d-1}{2}}\Gamma\left(\frac{d+1}{2}\right)}
\left(\vphantom{\int \limits_{\sqrt{\alpha}}^1}\right.
&\frac{4(d-1)}{\pi}\int \limits_{\sqrt{\alpha}}^1 \tau^{d-2}\arctan\left(\sqrt{(1-\alpha\tau^{-2})\left(\tau^{-2}-1\right)} \right)\dr\tau
\\ 
&+
\alpha^{\frac{d-1}2}
+
d-1
\left.\vphantom{\int \limits_{\sqrt{\alpha}}^1}\right)\,,
\end{split}
\end{equation}
and
\begin{equation}\label{eq:maintheoremfree}
\begin{split}
b_\free=
\frac{\mu^{\frac{1-d}{2}}}{2^{d+1}\pi^{\frac{d-1}{2}}\Gamma\left(\frac{d+1}{2}\right)}
\left(
\vphantom{\int\limits_{\sqrt{\alpha}}^1 \tau^{d-2}\arctan\left(\frac{\left(\tau^{-2}-2\right)^2}{4\sqrt{(1-\alpha\tau^{-2})\left(\tau^{-2}-1\right)}} \right)}
\right.
&\frac{4(d-1)}{\pi}\int\limits_{\sqrt{\alpha}}^1 \tau^{d-2}\arctan\left(\frac{\left(\tau^{-2}-2\right)^2}{4\sqrt{(1-\alpha\tau^{-2})\left(\tau^{-2}-1\right)}} \right)\dr\tau
\\ 
&
+
\alpha^{\frac{d-1}2}
+
d-5
+
4\,\gamma_R^{1-d}
\left.
\vphantom{\int\limits_{\sqrt{\alpha}}^1 \tau^{d-2}\arctan\left(\frac{\left(\tau^{-2}-2\right)^2}{4\sqrt{(1-\alpha\tau^{-2})\left(\tau^{-2}-1\right)}} \right)}
\right),
\end{split}
\end{equation}
respectively, 
where $\alpha$ is given by \eqref{eq:alpha}
and
$\gamma_R$ is given by \eqref{eq:gammaR}.
\end{thm}

Theorem~\ref{thm:maintheorem} will be proved in \S\S\ref{sec:2termElast}--\ref{sec:secondInvSub} by implementing the algorithm described in \S\ref{sec:algorithm}.

\begin{rem}
The bound  \eqref{eq:alphabound} guarantees that
\eqref{eq:maintheoremDir} and \eqref{eq:maintheoremfree} are well-defined and real.
\end{rem}

We show the appropriately rescaled (for the ease of comparison and to remove the explicit dependence on $\mu$) coefficients $b_\Dir$ and $b_\free$  as functions of $\alpha$ in Figure \ref{fig:bs}.

\begin{figure}[htbp]
\begin{center}
\includegraphics{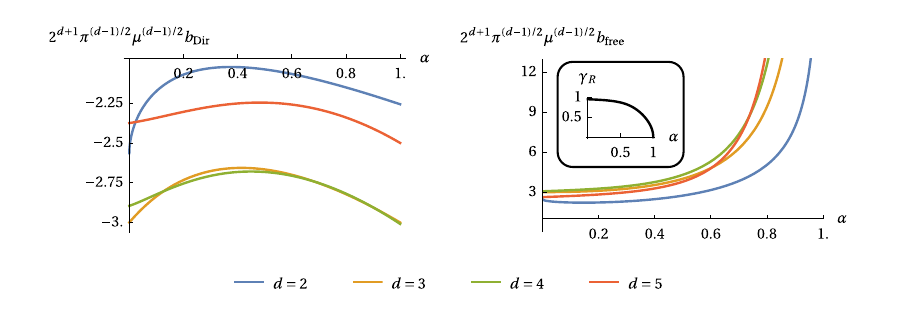}
\caption{The appropriately rescaled coefficients $b_\Dir$ (left image) and $b_\free$ (right image)  as functions of the parameter $\alpha$ for dimensions two to five. The inset in the right image shows the graph of $\gamma_R$ as a function of $\alpha$. We note that $\lim_{\alpha\to 0^+}\gamma_R\approx 0.9553$ and $\lim_{\alpha\to 1^-}\gamma_R=0$.}
\label{fig:bs}
\end{center}
\end{figure}

As it turns out, in odd dimensions the integrals in formulae \eqref{eq:maintheoremDir} and \eqref{eq:maintheoremfree} can be evaluated explicitly.

\begin{thm}
\label{thm:formulaeOddDim}
In dimension $d=2k+1$, $k=1,2,\ldots$, formulae \eqref{eq:maintheoremDir} and \eqref{eq:maintheoremfree} can be rewritten as
\begin{equation}\label{eq:maintheoremDirOdd}
b_\Dir=
-\frac{\mu^k}{2^{2k+2}k!\pi^k}
\left(
\left.
\frac{2}{k!}\frac{\dr^k}{\dr t^k}
\left(
\frac{2t-\frac{\alpha+1}{\alpha}}{t-\frac{\alpha+1}{\alpha}}\frac{1}{\sqrt{(1-\alpha t)(1-t)}}
\right)
\right|_{t=0}
-
2\left(\frac{\alpha}{\alpha+1}\right)^k
+
\alpha^k
+2k
\right),
\end{equation}
and
\begin{multline}
\label{eq:maintheoremfreeOdd}
b_\free=\frac{\mu^k}{2^{2k+2}k!\pi^k}
\left(
-\frac{8}{k!}\left.\frac{\dr^{k}}{\dr t^k}
\left(
\frac{\left(2\alpha t^2+(\alpha-3)t+2(1-\alpha)\right)\left(4\sqrt{(1-\alpha t)\left(1-t\right)}+(t-2)^2\right)}{(t-2)\left(t^3-8t^2+8(3-2\alpha)t+16(\alpha-1)\right)\sqrt{(1-\alpha t)\left(1-t\right)}} 
\right)
\right|_{t=0}
\right.
\\
\left.
\vphantom{-\frac{8}{k!}\left.\frac{\dr^{k}}{\dr t^k}
\left(
\frac{\left(2\alpha t^2+(\alpha-3)t+2(1-\alpha)\right)\left(4\sqrt{(1-\alpha t)\left(1-t\right)}+(t-2)^2\right)}{(t-2)\left(t^3-8t^2+8(3-2\alpha)t+16(\alpha-1)\right)\sqrt{(1-\alpha t)\left(1-t\right)}} 
\right)
\right|_{t=0}}
-
\alpha^k
+
2\left(k+2^{2-k}-1\right)
\right).
\end{multline}
\end{thm}
The proof of Theorem~\ref{thm:formulaeOddDim} is given in Appendix~\ref{appendix:proof}.  

We list, in Tables \ref{table:1} and \ref{table:2}, the explicit expressions for $b_\aleph$, $\aleph\in\{\Dir,\free\}$, for the first few odd dimensions.

\begin{table}[htpb]
\centering
\begin{tabular}{@{}>{$}c<{$}>{$}c<{$}@{}}
\toprule
d & b_{\Dir} 
  \\\midrule
3&-\frac{\left(2 \alpha ^2+\alpha +3\right) \mu }{16 \pi  (\alpha +1)} \\\addlinespace[1ex]
5&-\frac{\left(7 \alpha ^4+12 \alpha ^3+6 \alpha ^2+36 \alpha +19\right) \mu ^2}{512 \pi
   ^2 (\alpha +1)^2} \\\addlinespace[1ex]
7&-\frac{\left(13 \alpha ^6+36 \alpha ^5+27 \alpha ^4+16 \alpha ^3+147 \alpha ^2+156
   \alpha +53\right) \mu ^3}{12288 \pi ^3 (\alpha +1)^3} \\\addlinespace[1ex]
9& -\frac{\left(99 \alpha ^8+376 \alpha ^7+500 \alpha ^6+200 \alpha ^5+146 \alpha ^4+1992
   \alpha ^3+3188 \alpha ^2+2168 \alpha +547\right) \mu ^4}{1572864 \pi ^4 (\alpha +1)^4}
   \\\addlinespace[1ex]
\bottomrule
\addlinespace[1ex]
\end{tabular}
\caption{The coefficient $b_\Dir$ for odd dimensions.}\label{table:1}  
\end{table}

\begin{table}[htpb]
\centering
\begin{tabular}{@{}>{$}c<{$}>{$}c<{$}@{}}
\toprule
d & b_{\free} 
  \\\midrule
3& \frac{\left(2 \alpha ^2-3 \alpha +3\right) \mu }{16 \pi  (1-\alpha)} \\\addlinespace[1ex]
5&  \frac{\left(-7 \alpha ^4+12 \alpha ^3+14 \alpha ^2-36 \alpha +21\right) \mu ^2}{512 \pi
   ^2 (1-\alpha)^2} \\\addlinespace[1ex]
7&  \frac{\left(13 \alpha ^6-36 \alpha ^5+30 \alpha ^4-56 \alpha ^3+159 \alpha ^2-168
   \alpha +62\right) \mu ^3}{12288 \pi ^3 (1-\alpha)^3} \\\addlinespace[1ex]
9& \frac{\left(-99 \alpha ^8+376 \alpha ^7-516 \alpha ^6+296 \alpha ^5+470 \alpha ^4-2200
   \alpha ^3+3468 \alpha ^2-2440 \alpha+661\right) \mu ^4}{1572864 \pi ^4 (1-\alpha)^4} \\\addlinespace[1ex]
\bottomrule
\addlinespace[1ex]
\end{tabular}
\caption{The coefficient $b_\free$ for odd dimensions.}\label{table:2}  
\end{table}

\begin{rem}\label{rem:various}\ 
\begin{enumerate}[(i)]
\item 
Integrals in formulae \eqref{eq:maintheoremDir} and \eqref{eq:maintheoremfree} can be evaluated explicitly in even dimensions as well, but in this case one ends up with complicated expressions involving elliptic integrals.  Given that the outcome would not be much simpler or more elegant than the original formulae \eqref{eq:maintheoremDir} and \eqref{eq:maintheoremfree}, we omit the explicit evaluation of the integrals in even dimensions.
\item
In dimensions $d=2$ and $d=3$, formulae for the second Weyl coefficient for the operator of linear elasticity both for Dirichlet and free boundary conditions are given in \cite[Section~6.3]{SaVa}. The formulae in \cite{SaVa} have been obtained by applying the algorithm described below in \S\ref{sec:algorithm}, but the level of detail therein is somewhat insufficient, with only the final expressions being provided, without any intermediate steps.  Our results, when specialised to $d=2$ and $d=3$, agree with those of \cite{SaVa} and allow one to recover these results whilst providing the detailed derivation missing in \cite{SaVa}.
\end{enumerate}
\end{rem}

\begin{rem}\label{rem:Liu}
Genquian Liu  \cite{Liu} claims to have obtained formulae for $\widetilde b_\Dir$ and $\widetilde b_\free$.  However, the strategy adopted in \cite{Liu} is fundamentally flawed, because the ``method of images'' does not work for the operator of linear elasticity. Consequently, the main results from \cite{Liu} are wrong.

We postpone a more detailed discussions of \cite{Liu}, including the limitations of the method of images and a brief historical account of the development of the subject,  until Appendix~\ref{appendix:Liu}. Below, we provide a preliminary ``experimental'' comparison of our results and those in \cite{Liu}.

Essentially, \cite{Liu} aims to deduce the expression for the second asymptotic heat trace expansion coefficient $\widetilde b_\Dir$ in the Dirichlet case, as well as a corresponding expression in the case of the boundary conditions \cite[formula (1.5)]{Liu} (called there the ``Neumann''\footnote{The quotation marks are ours.} conditions) which in our notation\footnote{Note that our notation often differs from that of \cite{Liu}.} read
\begin{equation}\label{eq:Neu} 
n^\beta \nabla_\beta u^\alpha=0.
\end{equation}
We observe that the boundary conditions \eqref{eq:Neu} are \emph{not} self-adjoint for \eqref{eq:Lspec},  as easily seen by simple integration by parts.
Therefore, it is hard to assign a meaning to Liu's result in this case \cite[Theorem 1.1, the lower sign version of formula (1.10)]{Liu}. Nevertheless, even if one interprets the ``Neumann'' conditions \eqref{eq:Neu}  as our free boundary conditions \eqref{eq:trac}, as the author suggests in a post-publication revision \cite[formula (1.3)]{LiuUpdate}, the result of  \cite{Liu} in the free boundary case remains wrong.

\

For the sake of clarity, let us compare the results in the case of Dirichlet boundary conditions only.
The main result of \cite{Liu} in the Dirichlet case is \cite[Theorem 1.1, the upper sign version of formula (1.10)]{Liu}, which correctly states the coefficient $\widetilde a$ (cf. our formulae \eqref{eq:tildeC} and \eqref{eq:C}), and also states, in our notation, that
\begin{equation}\label{eq:BLiu}
\widetilde b^\Liu_\Dir:=-\frac{1}{4 (4\pi)^{(d-1)/2}}\left(\frac{d-1}{\mu^{(d-1)/2}}+\frac{1}{(\lambda+2\mu)^{(d-1)/2}}\right).
\end{equation}
This also implies, by \eqref{eq:BvstildeB},
\begin{equation}\label{eq:tildeBLiu}
b^\Liu_\Dir=\frac{\tilde b^\Liu_\Dir}{\Gamma\left(1+\frac{d-1}{2}\right)}=
-\frac{\mu^{\frac{1-d}{2}}}{2^{d+1}\pi^{\frac{d-1}{2}}\Gamma\left(\frac{d+1}{2}\right)}\left(\alpha^{(d-1)/2}+d-1\right),
\end{equation}
which differs from our expression \eqref{eq:maintheoremDir} by a missing integral term. 

For the reasons explained in Appendix~\ref{appendix:Liu}, formula \eqref{eq:BLiu} is incorrect.  We illustrate this by first showing, in Figure \ref{fig:bLiutob}, the 
ratio of the coefficient $b^\Liu_\Dir$ and our coefficient $b_\Dir$. 

\begin{figure}[htbp]
\begin{center}
\includegraphics{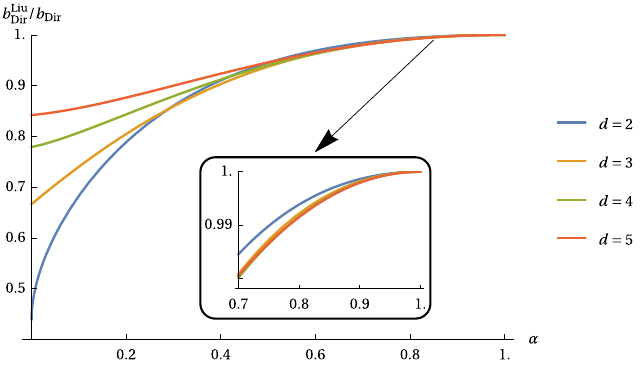}
\caption{The ratio $b^\Liu_\Dir/b_\Dir$ for different dimensions, shown as functions of $\alpha$.}
\label{fig:bLiutob}
\end{center}
\end{figure}

This ratio depends only on the dimension $d$ and the parameter $\alpha$. For each $d$, the ratio is monotone increasing in $\alpha$ (and is therefore monotone decreasing in $\lambda/\mu$). As $\alpha\to 1^-$ (or $\lambda\to -\mu^+$),  $b^\Liu_\Dir/b_\Dir\to 1^-$ in any dimension, see inset to Figure \ref{fig:bLiutob}. Thus, for the smallest possible values of the Lam\'e coefficient $\lambda$, Liu's asymptotic formula would produce an almost correct result, however  the error  would become more and more noticeable as $\lambda/\mu$ gets large. 

We illustrate this phenomenon ``experimentally'' in Figure \ref{fig:sq} where we take $\Omega\subset\mathbb{R}^2$ to be the unit square. Neither the Dirichlet nor the free boundary problem in this case can be solved by separation of variables, so we find the eigenvalues using the finite element package \texttt{FreeFEM} \cite{FreeFem}. As $\Vol_2(\Omega)=1$ and  $\Vol_1(\partial{\Omega})=4$, \eqref{eq:N2term} in the Dirichlet case may be interpreted as
\[
\mathcal{N}_\Dir(\Lambda)-a\Lambda\approx  4b_\Dir\sqrt{\Lambda}
\]
for sufficiently large $\Lambda$, and we compare the numerically computed left-hand sides with the right-hand sides given by our expression \eqref{eq:maintheoremDir} and  Liu's expression \eqref{eq:tildeBLiu}. As we have predicted, for $\lambda=-1/2$ both asymptotic formulae give a good agreement with the numerics, however for larger values of $\lambda$ our formulae match the actual eigenvalue counting functions exceptionally well, whereas Liu's ones are obviously incorrect. 

\begin{figure}[htbp]
\begin{center}
\includegraphics[width=\textwidth]{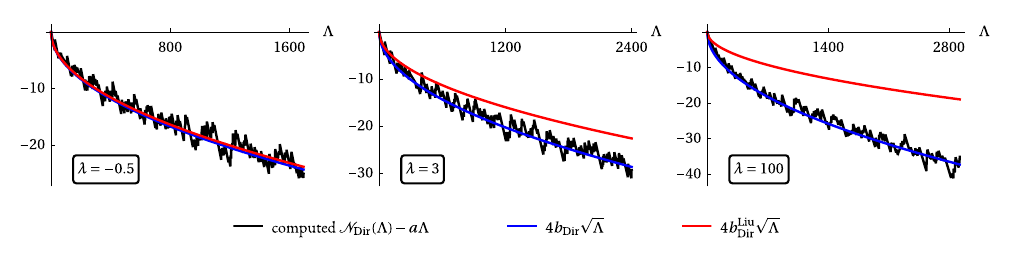}
\caption{The Dirichlet problem for the unit square. The second Weyl terms, both Liu's and ours, are compared to the actual numerically computed counting functions. In all three images we take $\mu=1$.}
\label{fig:sq}
\end{center}
\end{figure}

Of course, the boundary of a square is not smooth, only piecewise smooth, but this does not cause problems because this case is covered by \cite[Theorem 1]{Vas86}. Furthermore, \cite[Theorem 2]{Vas86} guarantees that sufficient conditions ensuring the validity of two-term asymptotic expansions \eqref{eq:N2term} are satisfied.
\end{rem}

For an additional illustration of the validity of our asymptotics in the free boundary case, see Figure \ref{fig:sqfree}. For further examples, both in the Dirichlet and the free boundary case, see Appendices \ref{sec:ex2d} and \ref{sec:ex3d}.

\begin{figure}[htbp]
\begin{center}
\includegraphics[width=\textwidth]{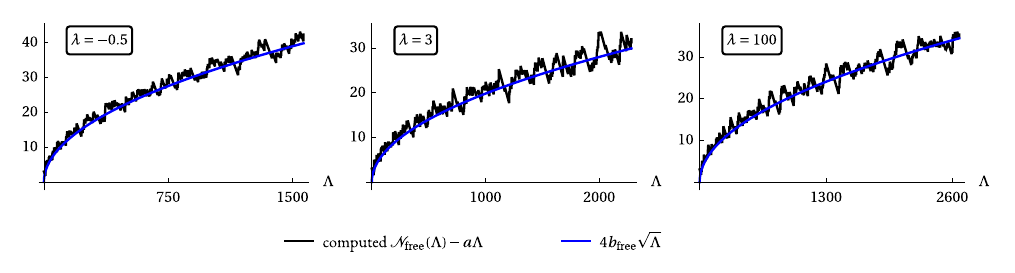}
\caption{The free boundary problem for the unit square. The second Weyl terms are compared to the actual numerically computed counting functions. In all three images we take $\mu=1$.}
\label{fig:sqfree}
\end{center}
\end{figure}

\section{Second Weyl coefficient for systems: an algorithm}
\label{sec:algorithm}

In this section we provide an algorithm for the determination of the second Weyl coefficient for more general elliptic systems. The algorithm given below is not new and appeared in \cite{Vas84,Vas86} as well as in \cite{SaVa} for scalar operators, with \cite[\S 6]{Vas84} briefly outlining the changes needed to adapt the results to systems.  However,  \cite{Vas84,Vas86} are not widely known and their English translations are somewhat unclear; therefore, we reproduce the algorithm here in a self-contained fashion and for matrix operators, for the reader's convenience.  In the next section we will explicitly implement the algorithm for $\mathcal{L}_\Dir$ and $\mathcal{L}_\free$.

Let $\mathcal{A}$ be a formally self-adjoint elliptic 
$m\times m$ differential operator of even order $2s$, semibounded from below.  Consider the spectral problem
\begin{equation}
\label{eq:SpecProbA}
\mathcal{A}\mathbf{u}=\Lambda\mathbf{u},
\end{equation}
\begin{equation}
\label{eq:SpecProbAbc}
\left.\mathcal{B}_j\mathbf{u}\right|_{\partial\Omega}=0, \qquad j=1,\dots, ms,
\end{equation}
where the $\mathcal{B}_j$'s are differential operators implementing self-adjoint boundary conditions of the Shapiro--Lopatinski type.  

It is well-known that the spectrum of \eqref{eq:SpecProbA}, \eqref{eq:SpecProbAbc} is discrete. Let us denote by
\[
\Lambda_n^{\mathcal{A}, \mathcal{B}}, \qquad n\in \mathbb{N},
\]
the eigenvalues of \eqref{eq:SpecProbA}, \eqref{eq:SpecProbAbc}, with account of multiplicity, and let
\begin{equation}
\label{eq:countingfnAB}
\mathcal{N}^{\mathcal{A}, \mathcal{B}}(\Lambda):=\#\left\{n: \ \Lambda_n^{\mathcal{A}, \mathcal{B}}<\Lambda\right\}
\end{equation}
be the corresponding eigenvalue counting function.

\

In a neighbourhood of the boundary $\partial \Omega$ we introduce local coordinates
\begin{equation}
\label{eq:coordPos}
x=(x',z),
\qquad \partial\Omega=\{z=0\},\quad z=\operatorname{dist}(x,\partial\Omega),
\end{equation}
so that $z>0$ for $x\in \Omega^\circ$, where $\Omega^\circ$ is the interior of $\Omega$.
We will also adopt the notation
\begin{equation}
\label{eq:coordMom}
\xi=(\xi',\zeta).
\end{equation}

Let $\mathcal{A}_\mathrm{prin}(x,\xi)$ be the principal symbol of $\mathcal{A}$ and suppose that $\xi\ne0$. Let $\tilde h_1(x,\xi)$, \ldots, $\tilde h_{\tilde m}(x,\xi)$ be the distinct eigenvalues of $\mathcal{A}_\mathrm{prin}(x,\xi)$ enumerated in increasing order.  Here $\tilde{m}=\tilde{m}(x,\xi)$ is a positive integer smaller than or equal to $m$.

\begin{assumption}
\label{ass:ConstMul}
The eigenvalues $\tilde h_k(x,\xi)$, $k=1,\ldots, \tilde m$, have constant multiplicities. In particular, the quantity $\tilde m$ is constant, independent of $(x,\xi)$. 
\end{assumption}

We will see in \S\ref{sec:2termElast} that the above assumption is satisfied for the operator of linear elasticity $\mathcal{L}$.

\begin{thm}[{\cite[Theorem~6.1]{Vas84}}]
Suppose that $(\Omega, g)$ is such that the corresponding billiards is neither \emph{dead-end} nor \emph{absolutely periodic}. Then the eigenvalue counting function \eqref{eq:countingfnAB} admits a two-term asymptotic expansion
\begin{equation*}
\mathcal{N}^{\mathcal{A},\mathcal{B}}(\Lambda)= A\,\Lambda^{\frac{d}{2s}}+ B_\mathcal{B}\Lambda^{\frac{d-1}{2s}}+o\left(\Lambda^{\frac{d-1}{2s}}\right)\quad\text{as}\quad\Lambda\to+\infty
\end{equation*}
for some real constants $A$ and $B_\mathcal{B}$. Furthermore:
\begin{enumerate}
\item[(a)]
The first Weyl coefficient $A$ is given by
\begin{equation*}
\label{eq:FirstWeylGeneral}
A=\frac{1}{(2\pi)^d}\int_{T^*\Omega}  n(x,\xi,1) \ \dr x\,\dr\xi,
\end{equation*}
where $n(x,\xi,\Lambda)$ is the eigenvalue counting function for the matrix-function $\mathcal{A}_\mathrm{prin}(x,\xi)$\footnote{That is, for each $(x,\xi)\in T^*\Omega$, the quantity $n(x,\xi,\Lambda)$ is the number of eigenvalues less than  $\Lambda$, with account of multiplicity, of the $m\times m$ matrix $\mathcal{A}_\mathrm{prin}(x,\xi)$.}.

\item[(b)]
The second Weyl coefficient $B_\mathcal{B}$ is given by
\begin{equation}
\label{eq:SecondWeylGeneral}
B_\mathcal{B}=\frac{1}{(2\pi)^{d-1}} \int_{T^*\partial\Omega}
\operatorname{shift}(x',\xi',1) \,\dr x' \,\dr\xi'\,,
\end{equation}
where the \emph{spectral shift function} is defined in accordance with
\begin{equation*}
\operatorname{shift}(x',\xi',\Lambda):=\frac{\varphi(x',\xi',\Lambda)}{2\pi}+N(x',\xi',\Lambda)\,,
\end{equation*}
and the \emph{phase shift} $\varphi(x',\xi',\Lambda)$ and the \emph{one-dimensional counting function} $N(x',\xi',\Lambda)$ are determined via the algorithm given below.
\end{enumerate}
\end{thm}

{\bf Step 1}: {\em One-dimensional spectral problem}. Construct the ordinary differential operators $\mathcal{A}'$ and $\mathcal{B}'_j$ from the partial differential operators $\mathcal{A}$ and $\mathcal{B}_j$ as follows:
\begin{enumerate}[(i)]
\item
retain only the terms containing the derivatives of the highest order in $\mathcal{A}$ and $\mathcal{B}_j$;

\item
replace partial derivatives along the boundary with $\ir$ times the corresponding component of momentum:
\[
\partial_{x'}\mapsto \ir\xi'\,;
\]

\item
evaluate all coefficients at $z=0$.
\end{enumerate}
The operators $\mathcal{A}'=\mathcal{A}'(x',\xi')$ and $\mathcal{B}'_j=\mathcal{B}'_j(x',\xi')$ are ordinary differential operators in the variable $z$ with coefficients depending on $x'$ and $\xi'$.

Consider the one-dimensional spectral problem
\begin{equation}
\label{eq:1dimProblA} 
\mathcal{A}'\mathbf{u}(z)=\Lambda\mathbf{u}(z),
\end{equation}
\begin{equation}
\label{eq:1dimProblBC} 
\left. \mathcal{B}'_j \mathbf{u}(z)\right|_{z=0}=0, \qquad j=1,\ldots,sm.
\end{equation}

\

{\bf Step 2}: {\em Thresholds and continuous spectrum}.
Suppose that $\xi'\ne0$. Let $h_k(\zeta)$, $k=1, \dots, \tilde m$, be the distinct eigenvalues of $(\mathcal{A}')_\mathrm{prin}(\zeta)$ enumerated in increasing order and let $m_k$ be their multiplicities,
so that
\[
\sum_{k=1}^{\tilde{m}} m_k=m.
\]
Clearly, for fixed $(x',\xi')$ we have
\[
h_k(\zeta)=\tilde h_k((x',0),(\xi',\zeta)), \qquad k=1, \dots, \tilde m\,.
\]
In what follows, up to and including Step 6, we suppress, for the sake of brevity, the dependence on $x'$ and $\xi'$.

Compute the \emph{thresholds} of the continuous spectrum, namely, nonnegative real numbers $\Lambda_*$ such that the equation
\begin{equation*}
h_k(\zeta)=\Lambda_*
\end{equation*}
in the variable $\zeta$ has a multiple real root for at least one $k\in\{1,\ldots, \tilde{m}\}$.  We enumerate the $\overline{m}$ thresholds
in increasing order
\[
\Lambda_*^{(1)}<\dots<\Lambda_*^{(\overline{m})}.
\]

The thresholds partition the continuous spectrum $[\Lambda_*^{(1)},+\infty)$ of the problem \eqref{eq:1dimProblA}, \eqref{eq:1dimProblBC} into $\overline{m}$ intervals 
\begin{equation*}
I^{(l)}:=
\begin{cases}
\left(\Lambda_*^{(l)}, \Lambda_*^{(l+1)}\right) & \text{for}\ l=1,\ldots,\overline{m}-1,\\[1ex]
\left(\Lambda_*^{(\overline{m})},  +\infty\right) & \text{for}\ l=\overline{m}.
\end{cases}
\end{equation*}
For $\Lambda\in I^{(l)}$, let $k_\mathrm{max}^{(l)}$ be the largest $k$ for which the equation
\begin{equation}
\label{eq:zetaRoots}
h_k(\zeta)=\Lambda
\end{equation}
has real roots. Given a $k\in\{1,\ldots,k_\mathrm{max}^{(l)}\}$, let $2q_k^{(l)}$ be the number of real roots\footnote{The number of such roots is independent of the choice of a particular $\Lambda\in I^{(l)}$.} of equation \eqref{eq:zetaRoots}.
We define the multiplicity of the continuous spectrum in $I^{(l)}$ as
\[
p^{(l)}:=\sum_{k=1}^{k_\mathrm{max}^{(l)}}m_k\,q_k^{(l)}\,.
\]

\

{\bf Step 3}: {\em Eigenfunctions of the continuous spectrum}.
At this step we suppress, for the sake of brevity, the dependence on $l$ and write $k_\mathrm{max}=k_\mathrm{max}^{(l)}$, $q_k=q_k^{(l)}$, $p=p^{(l)}$.
In each interval $I^{(l)}$ denote the real roots of \eqref{eq:zetaRoots} for a given $\,k=1,\ldots,k_\mathrm{max}\,$ by
\begin{equation*}
\zeta_{k,q}^{\pm}(\Lambda), \qquad q=1,\ldots,q_k\,,
\end{equation*}
where the superscript $\pm$ is chosen in such a way that
\[
\operatorname{sign}\left(\left.\frac{\dr h_k(\zeta)}{\dr\zeta}\right|_{\zeta=\zeta^\pm_{k,q}(\Lambda)}\right)=\pm 1,
\]
and the roots are ordered in accordance with
\[
\zeta_{k,1}^{-}(\Lambda)<\zeta_{k,1}^{+}(\Lambda)<\ldots<\zeta_{k,q_k}^{-}(\Lambda)<\zeta_{k,q_k}^{+}(\Lambda)\,.
\]

Let
\begin{equation*}
\mathbf{v}^{(j)}_{k}(\zeta), \qquad k=1,\ldots,\tilde{m},\quad j=1,\ldots m_k,
\end{equation*}
be orthonormal eigenvectors of $(\mathcal{A}')_\mathrm{prin}(\zeta)$ corresponding to the eigenvalues $h_k(\zeta)$.  Of course, these eigenvectors are not uniquely defined: there is a $\mathrm{U}(m_k)$ gauge freedom in their choice.

For given $k\in\{1,\ldots,k_\mathrm{max}^{(l)}\}$, $q\in\{1,\ldots,q_k\}$ and $@\in\{+,-\}$, \,let
\[
\mathbf{w}^@_{k,q,j}(\Lambda)
:=
\mathbf{v}^{(j)}_{k}(\zeta^@_{k,q}(\Lambda)),
\]
where the gauge is chosen so that
\[
\sum_{j=1}^{m_k}
\left \langle \mathbf{w}^@_{k,q,j}(\Lambda), \frac{\dr \mathbf{w}^@_{k,q,j}(\Lambda)}{\dr \Lambda}\right \rangle=0. 
\]
This defines each of the two orthonormal bases
$\mathbf{w}^+_{k,q,j}(\Lambda)\,$, $\,j=1,\ldots m_k\,$,
and
$\mathbf{w}^-_{k,q,j}(\Lambda)\,$, $\,j=1,\ldots m_k\,$,
uniquely modulo a composition of a rigid ($\Lambda$-independent) $\ \mathrm{U}(m_k)\ $ transformation and a $\Lambda$-dependent $\ \mathrm{SU}(m_k)\ $ transformation.

We seek \emph{eigenfunctions of the continuous spectrum} (generalised eigenfunctions) for the one-dimen\-sional spectral problem \eqref{eq:1dimProblA},  \eqref{eq:1dimProblBC} corresponding to $\Lambda\in I^{(l)}$ in the form
\begin{equation}\label{eq:efcontspec_general}
\begin{split}
\mathbf{u}(z;\Lambda)
=
\sum_{k=1}^{k_\mathrm{max}}
\sum_{q=1}^{q_k}
\sum_{j=1}^{m_k}
&\ \sum_{@\in\{+,-\}}
\ \frac{c_{k,q,j}^{@}}{\sqrt{@2\pi \left. (\dr h_k/\dr\zeta)\right|_{\zeta=\zeta^{@}_{k,q}(\Lambda)}}} \mathbf{w}^@_{k,q,j}(\Lambda)\,\er^{\ir\zeta^@_{k,q}(\Lambda)\,z}
\\
&+
\sum_{r=1}^{ms-p} c_r\, \mathbf{f}_r(z;\Lambda),
\end{split}
\end{equation}
where $\mathbf{f}_1$, \dots, $\mathbf{f}_{ms-p}$ are linearly independent solutions of \eqref{eq:1dimProblA} tending to $0$ as $z\to+\infty$, and the coefficients $c_{k,q,j}^{@}$ are not all zero.

The coefficients $c_{k,q,j}^{@}$ are called \emph{incoming} ($@=-$) and \emph{outgoing} ($@=+$) complex wave amplitudes.

\

{\bf Step 4}: {\em The scattering matrix}.
Requiring that \eqref{eq:efcontspec_general} satisfies the boundary conditions \eqref{eq:1dimProblBC} allows one to express the outgoing amplitudes $\mathbf{c}^+$ in terms of the incoming amplitudes $\mathbf{c}^-$. This defines the \emph{scattering matrix} $S^{(l)}(\Lambda)$, a $p^{(l)}\times p^{(l)}$ unitary matrix, via
\begin{equation*}
\mathbf{c}^+=
S^{(l)}(\Lambda)\,\mathbf{c}^-\,.
\end{equation*}
The order in which coefficients $c_{k,q,j}^{@}$ are arranged into $p^{(l)}$-dimensional columns $\mathbf{c}^{@}$ is unimportant.

\

{\bf Step 5}: {\em The phase shift}.
Compute the \emph{phase shift} $\varphi(\Lambda)$, defined in accordance with
\begin{equation}
\label{eq:SpectralShiftFn}
\varphi(\Lambda):=
\begin{cases}
0 & \text{for }\Lambda\le \Lambda_*^{(1)},\\
\arg\left( \det S^{(l)}(\Lambda) \right) + \mathfrak{s}^{(l)} & \text{for }\Lambda\in I^{(l)}.
\end{cases}
\end{equation}
The quantities $\mathfrak{s}^{(l)}$, $l=1,\dots,\overline{m}$, are some real constants whose role is to account for the fact that our construction of orthonormal bases for incoming and outgoing complex wave amplitudes involves a rigid ($\Lambda$-independent) unitary gauge degree of freedom, see Step 3 above.
The branch of the multi-valued function $\arg$ appearing in formula \eqref{eq:SpectralShiftFn} is assumed to be chosen in such a way that the phase shift $\varphi(\Lambda)$ is continuous in each interval $I^{(l)}$.

For each $l$, suppose that equation \eqref{eq:zetaRoots} with $\Lambda=\Lambda_*^{(l)}$ has a multiple real root for precisely one $k=k^{(l)}$, and that this multiple real root $\zeta=\zeta_*^{(l)}$ is unique and is a double root\footnote{This is the generic situation. In the general case the constants $\mathfrak{s}^{(l)}$ are obtained by integrating the trace of an appropriate generalised resolvent, see \cite[Eqn.~(1.14)]{Vas84}.}. Then the constants $\mathfrak{s}^{(l)}$ in \eqref{eq:SpectralShiftFn} are determined by requiring that the jumps of the phase shift at the thresholds satisfy
\begin{equation}
\label{eq:jumpsSSF_gen}
\frac{1}{\pi}\,\lim_{\epsilon\to 0^+} \left(\varphi(\Lambda_*^{(l)}+\epsilon)-\varphi(\Lambda_*^{(l)}-\epsilon) \right)= j_*^{(l)} -\frac{m_{k^{(l)}}}{2}\,,
\end{equation}
where $j_*^{(l)}$ is the number of linearly independent vectors $\mathbf{v}$ such that
\begin{equation}
\label{eq:oscillSolThre}
\mathbf{u}(z)=\mathbf{v} \,\er^{\ir\zeta_*^{(l)}z}+\mathbf{f}(z)
\end{equation}
is a solution of the one-dimensional problem \eqref{eq:1dimProblA}, \eqref{eq:1dimProblBC}, with $\mathbf{f}(z)=o(1)$ as $z \to +\infty$.

The threshold $\Lambda_*^{(l)}$ is called \emph{rigid} if $j_*^{(l)}=0$ and \emph{soft} if $j_*^{(l)}=m_{k^{(l)}}$. For rigid and soft thresholds formula \eqref{eq:jumpsSSF_gen} simplifies and reads
\begin{equation*}
\frac{1}{\pi}\,\lim_{\epsilon\to 0^+} \left(\varphi(\Lambda_*^{(l)}+\epsilon)-\varphi(\Lambda_*^{(l)}-\epsilon) \right)= \mp\frac{m_{k^{(l)}}}{2}\,,
\end{equation*}
minus for rigid and plus for soft.

\

{\bf Step 6}: {\em The one-dimensional counting function}.
Compute the \emph{one-dimensional counting function} 
\begin{equation*}
N(\Lambda):=\# \{ \text{eigenvalues of \eqref{eq:1dimProblA}, \eqref{eq:1dimProblBC} smaller than }\Lambda\}.
\end{equation*}

\

Application of Steps 1--6 of the above algorithm to the elasticity operator with Dirichlet or free boundary conditions, which will be done in the next three sections, gives
\begin{thm}
\label{thm:shiftDir}
\begin{equation}
\label{eq:shiftDir}
\mathrm{shift}_{\Dir}(\xi',\Lambda)
=
\begin{cases}
0 &\text{for}\ \Lambda\le \mu\|\xi'\|^2,\\
-\frac{1}{\pi}\arctan\left(\sqrt{\left(1-\frac{\Lambda}{\lambda+2\mu}\frac1{\|\xi'\|^2}\right)\left(\frac{\Lambda}{\mu}\frac1{\|\xi'\|^2}-1\right)} \right)-\frac{d-1}4&\text{for}\ \mu\|\xi'\|^2<\Lambda< (\lambda+2 \mu)\|\xi'\|^2,\\
-\frac{d}4&\text{for}\ \Lambda> (\lambda+2 \mu)\|\xi'\|^2,
\end{cases}
\end{equation}
\end{thm}
\noindent and
\begin{thm}
\label{thm:shiftfree}
\begin{equation}
\label{eq:shiftfree}
\mathrm{shift}_{\free}(\xi',\Lambda)
=
\begin{cases}
0 &\text{for}\ \Lambda< \mu\gamma_R^2\|\xi'\|^2,\\
1 &\text{for}\ \mu\gamma_R^2\|\xi'\|^2<\Lambda< \mu\|\xi'\|^2,\\
\frac{1}{\pi}
\arctan
\left(
\frac
{
\left(\frac{\Lambda}{\mu}\frac1{\|\xi'\|^2}-2\right)^2
}
{
4\,\sqrt{\left(1-\frac{\Lambda}{\lambda+2\mu}\frac1{\|\xi'\|^2}\right)\left(\frac{\Lambda}{\mu}\frac1{\|\xi'\|^2}-1\right)}
}
\right)+\frac{d-1}4&\text{for}\ \mu\|\xi'\|^2<\Lambda< (\lambda+2 \mu)\|\xi'\|^2,\\
\frac{d}4&\text{for}\ \Lambda> (\lambda+2 \mu)\|\xi'\|^2.
\end{cases}
\end{equation}
\end{thm}

In particular,  Theorem~\ref{thm:shiftDir} will follow from \eqref{eq:shiftequivpar2D}, Lemma~\ref{lem:shiftDirperp} and Lemma~\ref{lem:shiftDir2D}, whereas Theorem~\ref{thm:shiftfree} will follow from \eqref{eq:shiftequivpar2D}, Lemma~\ref{lem:shiftfreeperp} and Lemma~\ref{lem:shiftfree2D}.

Substituting \eqref{eq:shiftDir} and \eqref{eq:shiftfree} into 
\eqref{eq:SecondWeylGeneral} and performing straightforward algebraic manipulations we arrive at \eqref{eq:maintheoremDir} and \eqref{eq:maintheoremfree}, respectively, thus proving Theorem \ref{thm:maintheorem}. Note that
\begin{equation*}
B_\aleph=\mathrm{Vol}_{d-1}(\partial\Omega)\, b_\aleph, \qquad \aleph\in\{\Dir,\free\}\,.
\end{equation*}

\section{Second Weyl coefficients for linear elasticity: invariant subspaces}
\label{sec:2termElast}

In this and the next two sections we will compute the spectral shift function for the operator of linear elasticity on a Riemannian manifold with boundary of arbitrary dimension $d\ge 2$,  both for Dirichlet and free boundary conditions, by explicitly implementing the algorithm from \S \ref{sec:algorithm}. This will establish Theorems~\ref{thm:shiftDir} and~\ref{thm:shiftfree}.

In order to substantially simplify the calculations, we will turn some ideas of Dupuis--Mazo--Onsager \cite{Onsager} into a rigorous mathematical argument, in the spirit of \cite{part2}.  Namely, we will introduce two invariant subspaces for the elasticity operator compatible with the boundary conditions,  implement the algorithm in each invariant subspace separately, and combine the results in the end. 

\

As explained in \S\ref{sec:mainresults} (see Fact~\ref{fact:6}) it is sufficient to determine the second Weyl coefficients in the Euclidean setting, $g_{\alpha\beta}=\delta_{\alpha\beta}$. Furthermore, the construction presented in the beginning of \S\ref{sec:algorithm} (see formulae \eqref{eq:coordPos}, \eqref{eq:coordMom}) allows us to work in a Euclidean half-space. Hence, further on $x=(x^1,\ldots,x^d)$ are Cartesian coordinates, $x'=(x^1,\ldots,x^{d-1})$, $z=x^d$ and $\Omega=\{z\ge0\}$. Accordingly, we write $\xi=(\xi_1,\ldots,\xi_d)$, $\xi'=(\xi_1,\ldots,\xi_{d-1})$ and $\zeta=\xi_d$.

\

For starters, let us observe that the standard separation of variables leading to the one-dimensional problem 
\eqref{eq:1dimProblA}, \eqref{eq:1dimProblBC} can be achieved by seeking a solution of the form
\begin{equation*}
\er^{\ir \langle x', \xi'\rangle}\mathbf{u}(z)\,.
\end{equation*}

Next, suppose we have fixed $\xi'\in \mathbb{R}^{d-1}\setminus\{0\}$. Consider the pair of constant $d$-dimensional columns
\begin{equation*}
\frac1{\|\xi'\|}\begin{pmatrix}
\xi'
\\
0
\end{pmatrix}, 
\quad
\begin{pmatrix}
0'
\\
1
\end{pmatrix}\,,
\end{equation*}
where $0'$ stands for the $(d-1)$-dimensional column of zeros.
These define a two-dimensional plane
\[
P:=\operatorname{span}\left\{
\frac1{\|\xi'\|}\begin{pmatrix}
\xi'
\\
0
\end{pmatrix}\,, 
\
\begin{pmatrix}
0'
\\
1
\end{pmatrix}
\right\}\subset \mathbb{R}^d.
\]
 Let us denote by $\Pi$ the orthogonal projection onto $P$.  


Now, the principal symbol of the elasticity operator reads
\begin{equation}
\label{eq:Lprin}
\mathcal{L}_\mathrm{prin}(\xi)=(\mathcal{L}')_\mathrm{prin}(\zeta)=\mu\|\xi\|^2 I + (\lambda+\mu)\xi\xi^T.
\end{equation}
Formula \eqref{eq:Lprin} immediately implies that the eigenvalues of the principal symbol are 
\begin{equation}
\label{eq:h1elast}
\tilde h_1(\xi)=h_1(\zeta)=\mu\|\xi\|^2, \qquad \text{of multiplicity $m_1=d-1$},
\end{equation}
and
\begin{equation}
\label{eq:h2elast}
\tilde h_2(\xi)=h_2(\zeta)=(\lambda+2\mu)\|\xi\|^2, \qquad \text{of multiplicity $m_2=1$}.
\end{equation}
Formulae \eqref{eq:mulambda}, \eqref{eq:h1elast} and \eqref{eq:h2elast} imply that Assumption~\ref{ass:ConstMul} is satisfied.
The eigenspaces corresponding to \eqref{eq:h1elast} and \eqref{eq:h2elast} are
\begin{equation*}
(I-\|\xi\|^{-2}\xi\xi^T)\,\mathbb{R}^d \qquad \text{and}\qquad \operatorname{span}\{\xi\},
\end{equation*}
respectively. 

It is easy to see that $\xi\in P\,$, $\,P^\perp\subset (I-\|\xi\|^{-2}\xi\xi^T)\,\mathbb{R}^d$, and that $P$ and $P^\perp$ are invariant subspaces of $\mathcal{L}_\mathrm{prin}$. Furthermore,  $\left.\mathcal{L}_\mathrm{prin}\right|_P$ has two simple eigenvalues,  $(\lambda+2\mu)\|\xi\|^2$ and $\mu\|\xi\|^2$, whereas $\left.\mathcal{L}_\mathrm{prin}\right|_{P^\perp}$ has one eigenvalue $\mu\|\xi\|^2$ of multiplicity $d-2$.

\

The above decomposition can be lifted to the space of vector fields. We define
\begin{equation*}
V_\parallel:=\{\mathbf{u}\in C^\infty[0,+\infty):  \mathbf{u}=\Pi\,\mathbf{u}\}
\end{equation*}
and
\begin{equation*}
V_\perp:=\{\mathbf{u}\in C^\infty[0,+\infty): \mathbf{u}=(I-\Pi)\,\mathbf{u}\}\,.
\end{equation*}

Let 
\begin{equation}
\label{eq:L'}
\mathcal{L}'=\mu\left(\|\xi'\|^2-\frac{\dr^2}{\dr z^2}\right) I 
-
(\lambda+\mu)\begin{pmatrix}
\ir\xi'
\\
\frac{\dr}{\dr z}
\end{pmatrix}
\begin{pmatrix}
\ir\xi'
&
\frac{\dr}{\dr z}
\end{pmatrix}
\end{equation}
and
\begin{equation}
\label{eq:T'}
\mathcal{T}'
=
-\lambda
\begin{pmatrix}
0'
\\
1
\end{pmatrix}
\begin{pmatrix}
\ir\xi'
& \frac{\dr}{\dr z}
\end{pmatrix}
-\mu \left(I\frac{\dr}{\dr z}
+
\begin{pmatrix}
\ir\xi'
\\ 
\frac{\dr}{\dr z}
\end{pmatrix}
\begin{pmatrix}
0'
&
1
\end{pmatrix}
\right)
\end{equation}
be the one-dimensional operators associated with $\mathcal{L}$ and $\mathcal{T}$, respectively; recall that the latter are defined by formulae \eqref{eq:elasticity} and \eqref{eq:T}.
It turns out that the linear spaces $V_\parallel$ and $V_\perp$ are invariant subspaces of $\mathcal{L}'$ compatible with the boundary conditions.

\begin{lem}
\label{lem:invsub}
We have
\begin{enumerate}[(a)]
\item
\begin{equation}
\label{eq:invsub1}
\mathcal{L}' V_\parallel \subset V_\parallel\,,
\end{equation}
\begin{equation}
\label{eq:invsub2}
\mathcal{L}' V_\perp\subset V_\perp\,.
\end{equation}
\item
\begin{equation}
\label{eq:invsub3}
\left.\left(\mathcal{T}' V_\parallel\right)\right|_{z=0}\subset \left. V_\parallel\right|_{z=0}\,,
\end{equation}
\begin{equation}
\label{eq:invsub4}
\left.\left(\mathcal{T}' V_\perp\right)\right|_{z=0}\subset \left. V_\perp\right|_{z=0}\,.
\end{equation}
\end{enumerate}
\end{lem}

\begin{proof}
(a) A generic element of $V_\parallel$ reads
\begin{equation*}
\mathbf{u}_\parallel(z)=\frac{1}{\|\xi'\|}\begin{pmatrix}
\xi'
\\
0
\end{pmatrix}
f_1(z)
+
\begin{pmatrix}
0'
\\
1
\end{pmatrix} 
f_2(z)\,, \qquad f_1,f_2 \in C^\infty[0,+\infty).
\end{equation*}
Acting with \eqref{eq:L'} on $\mathbf{u}_\parallel(z)$ we get
\begin{equation*}
\begin{split}
(\mathcal{L}'\mathbf{u}_\parallel)(z)
&
=
\mu\|\xi'\|^2 \mathbf{u}_\parallel(z) 
-\mu\frac{1}{\|\xi'\|}
\begin{pmatrix}
\xi'
\\
0
\end{pmatrix}
f_1''(z)
-\mu
\begin{pmatrix}
0'
\\
1
\end{pmatrix} 
f_2''(z)
\\&-
\ir(\lambda+\mu)\|\xi'\|
\begin{pmatrix}
\ir\xi' f_1(z)
\\
f_1'(z)
\end{pmatrix}
-
(\lambda+\mu)
\begin{pmatrix}
\ir\xi' f_2'(z)
\\
f_2''(z)
\end{pmatrix}
\\
&
=
\frac{1}{\|\xi'\|}
\begin{pmatrix}
\xi'
\\
0
\end{pmatrix}
\left((\lambda+2\mu) \|\xi'\|^2f_1(z)-\mu f_1''(z) -\ir(\lambda+\mu)\|\xi'\|f_2'(z) \right)
\\
&
+
\begin{pmatrix}
0'
\\
1
\end{pmatrix} 
\left(\mu  \|\xi'\|^2 f_2(z)-(\lambda+2\mu) f_2''(z)-
\ir(\lambda+\mu)\|\xi'\|f_1'(z)\right)\,,
\end{split}
\end{equation*}
which is an element of $V_\parallel$.  This proves \eqref{eq:invsub1}.

A generic element of $V_\perp$ reads
\begin{equation*}
\mathbf{u}_\perp(z)=\sum_{j=1}^{d-2}\begin{pmatrix}
\psi_j
\\
0
\end{pmatrix}
f_j(z), \qquad f_j\in C^\infty[0,+\infty),
\end{equation*}
where $\psi_j$, $j=1,\dots,d-2$,  are linearly independent columns in $\mathbb{R}^{d-1}$ orthogonal to $\xi'$.
Acting with \eqref{eq:L'} on $\mathbf{u}_\perp(z)$ we get
\begin{equation*}
(\mathcal{L}'\mathbf{u}_\perp)(z)=\sum_{j=1}^{d-2}\begin{pmatrix}
\psi_j
\\
0
\end{pmatrix}
\,\mu \left(\|\xi'\|^2 f_j(z)-f_j''(z)\right)\,,
\end{equation*}
which is an element of $V_\perp$.  This proves \eqref{eq:invsub2}.

(b) Acting with \eqref{eq:T'} on $\mathbf{u}_\parallel(z)$ we get
\begin{equation*}
-\left.\left(\mathcal{T}'\mathbf{u}_\parallel\right)\right|_{z=0}=
\frac{1}{\|\xi'\|}
\begin{pmatrix}
\xi'
\\
0
\end{pmatrix}
\left(\mu f_1'(0)+\ir \mu \|\xi'\|f_2(0)\right)
+
\begin{pmatrix}
0'
\\
1
\end{pmatrix} 
\left( \ir\lambda\|\xi'\| f_1(0)+(\lambda+2\mu) f_2'(0) \right)\,,
\end{equation*}
from which one obtains \eqref{eq:invsub3}.

Acting with \eqref{eq:T'} on $\mathbf{u}_\perp(z)$ we get
\begin{equation*}
-\left.\left(\mathcal{T}'\mathbf{u}_\perp\right)\right|_{z=0}=
\sum_{j=1}^{d-2}\begin{pmatrix}
\psi_j
\\
0
\end{pmatrix}
\,\mu f_j'(0)\,,
\end{equation*}
which immediately implies \eqref{eq:invsub4}.
\end{proof}

Lemma~\ref{lem:invsub} implies, via a standard density argument,  that the operators $\mathcal{L}'_\aleph$, $\aleph\in\{\Dir, \free\}$,  decompose as
\begin{equation*}
\mathcal{L}'_{\aleph}=\mathcal{L}'_{\aleph,\perp}\oplus \mathcal{L}'_{\aleph,\parallel},
\end{equation*}
where
$\mathcal{L}'_{\aleph,\perp}:=\left.\mathcal{L}'_{\aleph}\right|_{(I-\Pi)D(\mathcal{L}'_{\aleph})}$ and $\mathcal{L}'_{\aleph,\parallel}:=\left.\mathcal{L}'_{\aleph}\right|_{\Pi D(\mathcal{L}'_{\aleph})}$,  $D(\mathcal{L}'_{\aleph})$ being the domain of $\mathcal{L}'_{\aleph}$.

It is then a straightforward consequence of the Spectral Theorem that we can compute the spectral shift function for $\mathcal{L}'_{\aleph,\perp}$ and  $\mathcal{L}'_{\aleph,\parallel}$ separately, and sum up the results in the end.  More formally, we have
\begin{equation*}
\mathrm{shift}_{\aleph}=\mathrm{shift}_{\aleph,\perp}+\mathrm{shift}_{\aleph,\parallel}, \qquad \aleph\in\{\Dir,\free\}.
\end{equation*}

\textbf{Additional simplification}: it suffices to implement our algorithm for the special case
\begin{equation}
\label{eq:specialxi}
\xi'=\begin{pmatrix}
0
\\
\vdots
\\
0
\\
1
\end{pmatrix}\in \mathbb{R}^{d-1}.
\end{equation}
The general case can then be recovered by rescaling the spectral parameter in the end,  in accordance with
\begin{equation*}
\Lambda \mapsto \frac{\Lambda}{\|\xi'\|^2}\,.
\end{equation*}
In the next two sections,  we assume \eqref{eq:specialxi}.

\section{First invariant subspace: normally polarised waves}
\label{sec:firstInvSub}

In this section we will compute the spectral shift functions $\mathrm{shift}_{\aleph,\perp}$ for the operators $\mathcal{L}'_{\aleph,\perp}$, $\aleph\in\{\Dir, \free\}$.

\subsection{Dirichlet boundary conditions}
\label{subsec:firstInvSubDir}
Consider the spectral problem
\begin{equation}
\label{eq:EPperpDir}
\mathcal{L}'_\perp \mathbf{u}_\perp= \mu\left(1-\frac{\dr^2}{\dr z^2}\right) \mathbf{u}_\perp=\Lambda \mathbf{u}_\perp,
\end{equation}
\begin{equation}
\label{eq:BCperpDir}
\mathbf{u}_\perp|_{z=0}=0\,.
\end{equation}
The goal of this subsection is to prove the following result.
\begin{lem}
\label{lem:shiftDirperp}
We have
\begin{equation*}
\varphi_{\Dir,\perp}(\Lambda)=-\frac{(d-2)\pi}{2}\,\chi_{[\mu,+\infty)}(\Lambda)
\end{equation*}
and
\begin{equation*}
N_{\Dir,\perp}(\Lambda)=0,
\end{equation*}
so that
\begin{equation*}
\mathrm{shift}_{\Dir,\perp}(\Lambda)=-\frac{d-2}{4}\,\chi_{[\mu,+\infty)}(\Lambda).
\end{equation*}
\end{lem}

Here and further on $\chi_A$ denotes the indicator function of a set $A\subset\mathbb{R}$.

We shall prove Lemma~\ref{lem:shiftDirperp} in several steps.

The principal symbol $(\mathcal{L}'_\perp)_\mathrm{prin}\,$, as a linear operator in $P^\perp$,  has only one eigenvalue
\begin{equation}
\label{eq:h1perp}
h_{1,\perp}(\zeta)=h_1(\zeta)=\mu(1+\zeta^2)
\end{equation}
of multiplicity $m_1^\perp=d-2$.
The eigenvalue \eqref{eq:h1perp} determines the threshold
\begin{equation}
\label{eq:thresholdperp}
\Lambda_*^{(1)}=\mu
\end{equation}
which, in turn, yields exponents
\begin{equation*}
\zeta_{1,1}^{\pm}(\Lambda)=\pm\sqrt{\frac{\Lambda}{\mu}-1}\,.
\end{equation*}

Therefore, the continuous spectrum of the operator $\mathcal{L}'_\perp$ contains a single interval $I^{(1)}_\perp:=(\mu, +\infty)$ and the multiplicity of the continuous spectrum on this interval is $p_\perp^{(1)}=1$.  For $\Lambda\in I^{(1)}_\perp$, the eigenfunctions of the continuous spectrum read
\begin{equation}
\label{eq:uperp}
\mathbf{u}_\perp(z;\Lambda)=\sum_{j=1}^{d-2}\mathbf{e}_j \left(c_j^+ \er^{\ir\sqrt{\frac{\Lambda}{\mu}-1}\,z}+c_j^- \er^{-\ir\sqrt{\frac{\Lambda}{\mu}-1}\,z}\right),
\end{equation}
where $(\mathbf{e}_j)_\alpha=\delta_{j \alpha}$.

Substituting \eqref{eq:uperp} into \eqref{eq:BCperpDir} we obtain
\begin{equation*}
S_{\Dir,\perp}^{(1)}(\Lambda)=-\mathrm{I}
\end{equation*}
which, in turn,  yields
\begin{equation}
\label{eq:detSDirperp}
\arg \det S_{\Dir,\perp}^{(1)}(\Lambda)=
\begin{cases}
0\quad\text{if $\,d\,$ is even,}
\\
\pi\quad\text{if $\,d\,$ is odd.}
\end{cases}
\end{equation}

\begin{lem}
\label{lem:thresholdDir1}
The threshold \eqref{eq:thresholdperp} for the problem \eqref{eq:EPperpDir}, \eqref{eq:BCperpDir} is rigid.
\end{lem}
\begin{proof}
It is straightforward to see that the problem \eqref{eq:EPperpDir}, \eqref{eq:BCperpDir} does not admit any solution of the form\footnote{Observe that the only real root of the equation $h_{1,\perp}(\zeta)=\Lambda_*^{(1)}$ is the double root $\zeta=0$.}
\begin{equation}
\label{Dima 15 July label 1}
\begin{pmatrix}
c_1
\\
\vdots
\\
c_{d-2}
\\
0
\\
0
\end{pmatrix},
\end{equation}
$c_1, \dots, c_{d-2}\in \mathbb{C}\,$,
other than the trivial one, from which the claim follows.
\end{proof}
\begin{lem}
\label{lem:evDir1}
The problem \eqref{eq:EPperpDir}, \eqref{eq:BCperpDir} does not have eigenvalues.
\end{lem}
\begin{proof}
It is easy to see that the problem \eqref{eq:EPperpDir}, \eqref{eq:BCperpDir} does not admit eigenvalues for $\Lambda\ge\mu$, i.e.~eigenvalues embedded in the continuous spectrum.
Furthermore, for $\Lambda<\mu$ a straightforward substitution shows that the only solution of \eqref{eq:EPperpDir}, \eqref{eq:BCperpDir} of the form
\begin{equation}
\label{Dima 15 July label 2}
\mathbf{u}_\perp(z;\Lambda)=\sum_{j=1}^{d-2}\mathbf{e}_j  c_j \er^{-\sqrt{1-\frac{\Lambda}{\mu}}\,z}
\end{equation}
is the trivial one. This concludes the proof.
\end{proof}

Combining formula \eqref{eq:detSDirperp},  Lemma~\ref{lem:thresholdDir1} and Lemma~\ref{lem:evDir1} one obtains Lemma~\ref{lem:shiftDirperp}.

\subsection{Free boundary conditions}
\label{subsec:firstInvSubfree}

Consider the spectral problem
\begin{equation}
\label{eq:EPperpfree}
\mathcal{L}'_\perp \mathbf{u}_\perp=\mu\left(1-\frac{\dr^2}{\dr z^2}\right) \mathbf{u}_\perp=\Lambda \mathbf{u}_\perp,
\end{equation}
\begin{equation}
\label{eq:BCperpfree}
\mathcal{T}'\mathbf{u}_\perp|_{z=0}=-\mu \,\mathbf{u}_\perp'(0)=0\,.
\end{equation}
The goal of this subsection is to prove the following result.
\begin{lem}
\label{lem:shiftfreeperp}
We have
\begin{equation*}
\varphi_{\free,\perp}(\Lambda)=\frac{(d-2)\pi}{2}\,\chi_{[\mu,+\infty)}(\Lambda)
\end{equation*}
and
\begin{equation*}
N_{\free,\perp}(\Lambda)=0,
\end{equation*}
so that
\begin{equation*}
\mathrm{shift}_{\free,\perp}(\Lambda)=\frac{d-2}{4}\,\chi_{[\mu,+\infty)}(\Lambda).
\end{equation*}
\end{lem}

Formulae \eqref{eq:h1perp}--\eqref{eq:uperp} apply unchanged to the free boundary case.  Substituting \eqref{eq:uperp} into \eqref{eq:BCperpfree} we obtain
\begin{equation*}
S_{\free,\perp}^{(1)}(\Lambda)=\mathrm{I}
\end{equation*}
which, in turn,  yields
\begin{equation}
\label{eq:detSfreeperp}
\arg \det S_{\free,\perp}^{(1)}(\Lambda) =0.
\end{equation}

\begin{lem}
\label{lem:thresholdfree1}
The threshold \eqref{eq:thresholdperp} for the problem \eqref{eq:EPperpfree}, \eqref{eq:BCperpfree} is soft.
\end{lem}

\begin{proof}
Result follows from the fact that
\eqref{Dima 15 July label 1}
is a solution of \eqref{eq:EPperpfree}, \eqref{eq:BCperpfree} for all $c_1, \dots, c_{d-2}\in \mathbb{C}\,$.
\end{proof}

\begin{lem}
\label{lem:evfree1}
The problem \eqref{eq:EPperpfree}, \eqref{eq:BCperpfree} does not have eigenvalues.
\end{lem}
\begin{proof}
It is easy to see that the problem \eqref{eq:EPperpfree}, \eqref{eq:BCperpfree} does not admit eigenvalues for $\Lambda\ge\mu$, i.e.~eigenvalues embedded in the continuous spectrum.
Furthermore, for $\Lambda<\mu$ a straightforward substitution shows that the only solution of \eqref{eq:EPperpfree}, \eqref{eq:BCperpfree} of the form \eqref{Dima 15 July label 2} is the trivial one. This concludes the proof.
\end{proof}

Combining \eqref{eq:detSfreeperp},  Lemma~\ref{lem:thresholdfree1}, and Lemma~\ref{lem:evfree1}, one obtains Lemma~\ref{lem:shiftfreeperp}.

\section{Second invariant subspace: reduction to the two-dimensional case}
\label{sec:secondInvSub}

In this section we will compute the spectral shift functions $\mathrm{shift}_{\aleph,\parallel}$, $\aleph\in\{\Dir, \free\}$,  for the $\mathcal{L}'_\parallel$.

Calculations in the second invariant subspace are trickier, in that, unlike $\mathcal{L}'_\perp$,  the operator $\mathcal{L}'_\parallel$ is not diagonal. However, our decomposition into invariant subspaces implies the following

\begin{fact}
\label{fact:reduction2D}
Let us denote by $\mathcal{L}_\plane$ the operator of linear elasticity for $d=2$.
Then the spectral shift function for the problem
\begin{equation*}
\mathcal{L}'_\parallel \mathbf{u}_\parallel=\Lambda \mathbf{u}_\parallel,
\end{equation*}
with Dirichlet/free boundary conditions
coincides with the spectral shift function for the operator $\mathcal{L}'_\plane$ with the same boundary conditions. Namely,
\begin{equation}
\label{eq:shiftequivpar2D}
\mathrm{shift}_{\aleph,\parallel}(\Lambda)=\mathrm{shift}_{\aleph,\plane}(\Lambda), \qquad \aleph\in\{\Dir, \free\}.
\end{equation}
\end{fact}
Fact \eqref{fact:reduction2D} can be easily established by observing that, under assumption \eqref{eq:specialxi},  elements in the domain of $\mathcal{L}'_\parallel$ are of the form
\begin{equation*}
\mathbf{u}_\parallel(z)=
\begin{pmatrix}
0\\
\vdots
\\
0
\\
f_1(z)
\\
f_2(z)
\end{pmatrix}\,.
\end{equation*}

In the remainder of this section we will compute the spectral shift function for the operator of linear elasticity in dimension two.

The principal symbol $\left(\mathcal{L}'_\plane\right)_\mathrm{prin}$ has two simple eigenvalues\footnote{In our notation, $m_1=m_2=1$.}
\begin{equation*}
h_1(\zeta)=\mu(1+\zeta^2), \qquad h_2(\zeta)=(\lambda+2\mu)(1+\zeta^2)\,.
\end{equation*}
These give us the two thresholds
\begin{equation*}
\Lambda_*^{(1)}=\mu, \qquad \Lambda_*^{(2)}=\lambda+2\mu
\end{equation*}
and the corresponding exponents
\begin{equation*}
\zeta_1^{\pm}(\Lambda)=\pm\sqrt{\frac{\Lambda}{\mu}-1}\,,\qquad \zeta_2^{\pm}(\Lambda)=\pm\sqrt{\frac{\Lambda}{\lambda+2\mu}-1}\,,
\end{equation*}
so that the continuous spectrum $[\mu,+\infty)$ is partitioned into the two intervals
\begin{equation*}
I^{(1)}=(\mu,\lambda+2\mu), \qquad I^{(2)}=(\lambda+2\mu,+\infty)
\end{equation*}
of multiplicities $p^{(1)}=1$ and $p^{(2)}=2$, respectively.

The normalised eigenvectors of $\left(\mathcal{L}'_\plane\right)_\mathrm{prin}$ are
\begin{equation*}
\mathbf{v}_1(\zeta)=\frac{1}{\sqrt{1+\zeta^2}}\begin{pmatrix}
1
\\
\zeta
\end{pmatrix}\,,
\qquad 
\mathbf{v}_2(\zeta)=\frac{1}{\sqrt{1+\zeta^2}}
\begin{pmatrix}
-\zeta
\\
1
\end{pmatrix}\,.
\end{equation*}
Hence, the eigenfunctions of the continuous spectrum for $\Lambda$ in $I^{(1)}$ and $I^{(2)}$ read
\begin{equation}
\label{eq:EFessspec2D1}
\mathbf{u}(z;\Lambda)=\sum_{@\in\{+,-\}}  \frac{1}{\sqrt{@4\pi\mu \,\zeta_1^{@}(\Lambda)}}c_1^@\,\mathbf{v}_1(\zeta_1^{@}(\Lambda))\er^{\ir\zeta_1^{@}(\Lambda)\,z}+ c \,\mathbf{v}_2\left(\ir\sqrt{1-\frac {\Lambda}{\lambda+2\mu}}\right)\er^{-\sqrt{1-\frac {\Lambda}{\lambda+2\mu}}\,z}
\end{equation}
and
\begin{equation}\label{eq:EFessspec2D2}
\begin{split}
\mathbf{u}(z;\Lambda)=\sum_{@\in\{+,-\}}  &\left(
\frac{1}{\sqrt{4\pi\mu \,|\zeta_1^{@}(\Lambda)}|}c_1^@\,\mathbf{v}_1(\zeta_1^{@}(\Lambda))\er^{\ir\zeta_1^{@}(\Lambda)\,z} \right.\\
&\ +
\left.\frac{1}{\sqrt{4\pi(\lambda+2\mu) \,|\zeta_2^{@}(\Lambda)}|}c_2^@\,\mathbf{v}_2(\zeta_2^{@}(\Lambda))\er^{\ir\zeta_2^{@}(\Lambda)\,z}
\right),
\end{split}
\end{equation}
respectively.

\subsection{Dirichlet boundary conditions}
\label{subsec:secondInvSubDir}
Consider the spectral problem
\begin{equation}
\label{eq:EP2DDir}
\mathcal{L}_\plane' \mathbf{u}=\begin{pmatrix}
\lambda+2\mu- \mu\frac{\dr^2}{\dr z^2} & -\ir(\lambda+\mu) \frac{\dr}{\dr z}\\
-\ir(\lambda+\mu) \frac{\dr}{\dr z} & \mu -(\lambda+2\mu)\frac{\dr^2}{\dr z^2}\\
\end{pmatrix}
\begin{pmatrix}
f_1(z)\\
f_2(z)
\end{pmatrix}
=
\Lambda 
\begin{pmatrix}
f_1(z)\\
f_2(z)
\end{pmatrix},
\end{equation}
\begin{equation}
\label{eq:BC2DDir}
\begin{pmatrix}
f_1(0)\\
f_2(0)
\end{pmatrix}=\begin{pmatrix}
0
\\
0
\end{pmatrix}\,.
\end{equation}
The goal of this subsection is to prove the following result.
\begin{lem}
\label{lem:shiftDir2D}
We have
\begin{equation*}
\varphi_{\Dir,\plane}(\Lambda)
=
\begin{cases}
0 &\text{for}\ \Lambda<\mu,\\
-2\arctan\left(\sqrt{\left(1-\frac{\Lambda}{\lambda+2\mu}\right)\left(\frac{\Lambda}{\mu}-1\right)} \right)-\frac\pi2&\text{for}\ \mu<\Lambda< \lambda+2 \mu,\\
-\pi&\text{for}\ \Lambda> \lambda+2 \mu
\end{cases}
\end{equation*}
and
\begin{equation*}
N_{\Dir,\plane}(\Lambda)=0,
\end{equation*}
so that
\begin{equation*}
\mathrm{shift}_{\Dir,\plane}(\Lambda)
=
\begin{cases}
0 &\text{for}\ \Lambda<\mu,\\
-\frac{1}{\pi}\arctan\left(\sqrt{\left(1-\frac{\Lambda}{\lambda+2\mu}\right)\left(\frac{\Lambda}{\mu}-1\right)} \right)-\frac14&\text{for}\ \mu<\Lambda< \lambda+2 \mu,\\
-\frac12&\text{for}\ \Lambda> \lambda+2 \mu.
\end{cases}
\end{equation*}
\end{lem}

We shall prove Lemma~\ref{lem:shiftDir2D} in several steps.

Substituting \eqref{eq:EFessspec2D1} and \eqref{eq:EFessspec2D2} into \eqref{eq:BC2DDir} we get
\begin{equation*}
S_{\Dir,\plane}(\Lambda)
=
\begin{cases}
\frac{\sqrt{(1-\frac{\Lambda}{\lambda+2\mu})(\frac{\Lambda}{\mu}-1)}+\ir}{\sqrt{(1-\frac{\Lambda}{\lambda+2\mu})(\frac{\Lambda}{\mu}-1)}-\ir} & \text{for}\ \Lambda\in I^{(1)},\\[1.5em]
\begin{pmatrix}
\frac{\sigma^2-1}{\sigma^2+1} & -\frac{2\sigma}{\sigma^2+1} \\
\frac{2\sigma}{\sigma^2+1} & \frac{\sigma^2-1}{\sigma^2+1}
\end{pmatrix} &  \text{for}\ \Lambda\in I^{(2)},
\end{cases}
\end{equation*}
where $\sigma:=\left(\frac{\Lambda}{\mu}-1\right)^{1/4}\left(\frac{\Lambda}{\lambda+2\mu}-1\right)^{1/4}$. The above equation implies
\begin{equation}
\label{eq:argdet2DDir}
\arg \det S_{\Dir,\plane}(\Lambda)
=
\begin{cases}
-2\arctan\left(\sqrt{\left(1-\frac{\Lambda}{\lambda+2\mu}\right)\left(\frac{\Lambda}{\mu}-1\right)} \right)& \text{for}\ \Lambda\in I^{(1)},\\
0 &  \text{for}\ \Lambda\in I^{(2)}.
\end{cases}
\end{equation}

\begin{lem}
\label{lem:thresholds2Ddir}
The thresholds $\Lambda_*^{(1)}$ and $\Lambda_*^{(2)}$ for the problem \eqref{eq:EP2DDir}, \eqref{eq:BC2DDir} are rigid.
\end{lem}
\begin{proof}
Let us first consider $\Lambda_*^{(1)}$.  The general solution to \eqref{eq:EP2DDir} of the form \eqref{eq:oscillSolThre} reads
\begin{equation}
\label{eq:oscsol1}
c_1
\begin{pmatrix}
1\\
\ir\sqrt{\frac{\lambda+\mu}{\lambda+2\mu}}
\end{pmatrix}
\er^{-\sqrt{\frac{\lambda+\mu}{\lambda+2\mu}}\,z}
+
c_2
\begin{pmatrix}
0
\\
1
\end{pmatrix}.
\end{equation}
Substituting the above expression into \eqref{eq:BC2DDir} gives us $c_1=c_2=0$. Hence, $\Lambda_*^{(1)}$ is rigid.

Let us now examine $\Lambda_*^{(2)}$. The general solution to \eqref{eq:EP2DDir} of the form \eqref{eq:oscillSolThre} reads
\begin{equation}
\label{eq:oscsol2}
c
\begin{pmatrix}
1
\\
0
\end{pmatrix}.
\end{equation}
The latter can only satisfy the Dirichlet boundary conditions if $c=0$, which implies that $\Lambda_*^{(2)}$ is rigid as well.
\end{proof}

\begin{lem}
\label{lem:EV2Ddir}
The problem \eqref{eq:EP2DDir}, \eqref{eq:BC2DDir} does not have eigenvalues, either below or embedded in the continuous spectrum.
\end{lem}
\begin{proof}
Arguing as in the proof of Lemma~\ref{lem:thresholds2Ddir}, it is easy to see that thresholds are not eigenvalues.

For $\Lambda\in(0,\mu)$ we seek an eigenfunction of \eqref{eq:EP2DDir} in the form
\begin{equation}
\label{eq:proofEV2DdirEq1}
\mathbf{u}(z;\Lambda)=c_1\, \mathbf{v}_1\left(\ir\sqrt{1-\frac {\Lambda}{\mu}}\right)\er^{-\sqrt{1-\frac {\Lambda}{\mu}}\,z}
+
 c_2 \,\mathbf{v}_2\left(\ir\sqrt{1-\frac {\Lambda}{\lambda+2\mu}}\right)\er^{-\sqrt{1-\frac {\Lambda}{\lambda+2\mu}}\,z}\,.
\end{equation}
Substituting \eqref{eq:proofEV2DdirEq1} into \eqref{eq:BC2DDir} we get
\begin{equation*}
\begin{pmatrix}
1 & -\ir \sqrt{1-\frac {\Lambda}{\lambda+2\mu}}\\
\ir \sqrt{1-\frac {\Lambda}{\mu}} & 1
\end{pmatrix}
\begin{pmatrix}
c_1
\\
c_2
\end{pmatrix}
=
\begin{pmatrix}
0
\\
0
\end{pmatrix},
\end{equation*}
so that the characteristic equation in $(0,\mu)$ reads
\begin{equation*}
\chi(\Lambda)=1-\sqrt{1-\frac {\Lambda}{\lambda+2\mu}}\sqrt{1-\frac {\Lambda}{\mu}}=0.
\end{equation*}
The latter has no solutions in $(0,\mu)$.

For $\Lambda\in I^{(1)}$ we seek an eigenfunction in the form \eqref{eq:EFessspec2D1} with $c_1^\pm=0$, so that the solution is square-integrable:
\begin{equation}
\label{eq:proofEV2DdirEq4}
\mathbf{u}(z;\Lambda)=c \,\mathbf{v}_2\left(\ir\sqrt{1-\frac {\Lambda}{\lambda+2\mu}}\right)\er^{-\sqrt{1-\frac {\Lambda}{\lambda+2\mu}}\,z}.
\end{equation}
The latter satisfies \eqref{eq:BC2DDir} if and only if $c=0$, which means there are no eigenfunctions for $\Lambda\in I^{(1)}$, 

By looking at \eqref{eq:EFessspec2D2},  it is easy to see that there are no (square-integrable) eigenfunctions corresponding to $\Lambda\in I^{(2)}$, which completes the proof.
\end{proof}

Combining \eqref{eq:argdet2DDir}, Lemma~\ref{lem:thresholds2Ddir}, and Lemma~\ref{lem:EV2Ddir}, one obtains Lemma~\ref{lem:shiftDir2D}.

\subsection{Free boundary conditions}
\label{subsec:secondInvSubfree}

Consider the spectral problem
\begin{equation}
\label{eq:EP2Dfree}
\mathcal{L}_\plane' \mathbf{u}=\begin{pmatrix}
\lambda+2\mu- \mu\frac{\dr^2}{\dr z^2} & -\ir(\lambda+\mu) \frac{\dr}{\dr z}\\
-\ir(\lambda+\mu) \frac{\dr}{\dr z} & \mu -(\lambda+2\mu)\frac{\dr^2}{\dr z^2}\\
\end{pmatrix}
\begin{pmatrix}
f_1(z)\\
f_2(z)
\end{pmatrix}
=
\Lambda 
\begin{pmatrix}
f_1(z)\\
f_2(z)
\end{pmatrix},
\end{equation}
\begin{equation}
\label{eq:BC2Dfree}
-\left.\left(\mathcal{T}'\mathbf{u}\right)\right|_{z=0}=\begin{pmatrix}
\mu (f_1'(0)+\ir f_2(0))
\\
\ir\lambda
f_1(0)+(\lambda+2\mu)f_2'(0)
\end{pmatrix}
=
\begin{pmatrix}
0
\\
0
\end{pmatrix}\,.
\end{equation}
The goal of this subsection is to prove the following result.

\begin{lem}
\label{lem:shiftfree2D}
We have
\begin{equation*}
\varphi_{\free,\plane}(\Lambda)
=
\begin{cases}
0 &\text{for}\ \Lambda< \mu,\\
2
\arctan
\left(
\frac
{
\left(\frac{\Lambda}{\mu}-2\right)^2
}
{
4\,\sqrt{\left(1-\frac{\Lambda}{\lambda+2\mu}\right)\left(\frac{\Lambda}{\mu}-1\right)}
}
\right)-\frac{3\pi}2&\text{for}\ \mu<\Lambda< \lambda+2 \mu,\\
-\pi&\text{for}\ \Lambda> \lambda+2 \mu
\end{cases}
\end{equation*}
and
\begin{equation*}
N_{\free,\plane}(\Lambda)=\chi_{(\gamma_R^2\mu,+\infty)}(\Lambda),
\end{equation*}
where
$\gamma_R$ is given by \eqref{eq:gammaR},
so that
\begin{equation*}
\mathrm{shift}_{\free,\plane}(\Lambda)
=
\begin{cases}
0 &\text{for}\ \Lambda< \mu\gamma_R^2,\\
1 &\text{for}\ \mu\gamma_R^2<\Lambda< \mu,\\
\frac{1}{\pi}
\arctan
\left(
\frac
{
\left(\frac{\Lambda}{\mu}-2\right)^2
}
{
4\,\sqrt{\left(1-\frac{\Lambda}{\lambda+2\mu}\right)\left(\frac{\Lambda}{\mu}-1\right)}
}
\right)+\frac{1}{4}&\text{for}\ \mu<\Lambda< \lambda+2 \mu,\\
\frac12&\text{for}\ \Lambda> \lambda+2 \mu.
\end{cases}
\end{equation*}
\end{lem}

We shall prove Lemma~\ref{lem:shiftfree2D} in several steps.

Substituting \eqref{eq:EFessspec2D1} and \eqref{eq:EFessspec2D2} into \eqref{eq:BC2Dfree} we get
\begin{equation}
\label{eq:argdet2Dfree}
\arg \det S_{\free,\plane}(\Lambda)
=
\begin{cases}
2\arctan
\left(
\frac
{
\left(\frac{\Lambda}{\mu}-2\right)^2
}
{
4\,\sqrt{\left(1-\frac{\Lambda}{\lambda+2\mu}\right)\left(\frac{\Lambda}{\mu}-1\right)}
}
\right) & \text{for}\ \Lambda\in I^{(1)},\\
0 &  \text{for}\ \Lambda\in I^{(2)}.
\end{cases}
\end{equation}

We now have
\begin{lem}
\label{lem:thresholds2Dfree}
\begin{enumerate}\ 
\item[(a)]
The threshold $\Lambda_*^{(1)}$ for the problem \eqref{eq:EP2Dfree}, \eqref{eq:BC2Dfree} is rigid.  That is,
\begin{equation*}
j_*^{(1)}=0.
\end{equation*}
\item[(b)]
The threshold $\Lambda_*^{(2)}$ for the problem \eqref{eq:EP2Dfree}, \eqref{eq:BC2Dfree} is soft if $\lambda=0$ and rigid otherwise.  That is,
\begin{equation*}
j_*^{(2)}=
\begin{cases}
0 & \text{for }\lambda\ne 0,\\
1 & \text{for }\lambda=0.
\end{cases}
\end{equation*}
\end{enumerate}
\end{lem}
\begin{proof}
(a) Substituting \eqref{eq:oscsol1} into \eqref{eq:BC2Dfree} we obtain the linear system
\begin{equation*}
\begin{pmatrix}
-2\sqrt{\frac{\lambda+\mu}{\lambda+2\mu}} & \ir
\\
-\ir \mu&0
\end{pmatrix}
\begin{pmatrix}
c_1
\\
c_2
\end{pmatrix}
=
\begin{pmatrix}
0
\\
0
\end{pmatrix},
\end{equation*}
which has no nontrivial solutions. Indeed,
\begin{equation*}
\det \begin{pmatrix}
-2\sqrt{\frac{\lambda+\mu}{\lambda+2\mu}}c_1 & \ir
\\
-\ir \mu&0
\end{pmatrix}=-\mu<0\,.
\end{equation*}

(b) The claim follows at once by substituting \eqref{eq:oscsol2} into \eqref{eq:BC2Dfree}.
\end{proof}

\begin{lem}
\label{lem:EV2Dfree}
The problem \eqref{eq:EP2Dfree}, \eqref{eq:BC2Dfree} has precisely one eigenvalue
\begin{equation}
\label{eq:LambdaR}
\Lambda_R:=\mu \gamma_R^2\,,
\end{equation}
where $\gamma_R$ is given by \eqref{eq:gammaR}.
\end{lem}

\begin{proof}
Arguing as in the proof of Lemma~\ref{lem:thresholds2Dfree}, it is easy to see that thresholds are not eigenvalues. 

For $\Lambda\in (0,\mu)$, we seek an eigenfunction in the form \eqref{eq:proofEV2DdirEq1}. Substituting \eqref{eq:proofEV2DdirEq1} into \eqref{eq:BC2Dfree} we get
\begin{equation*}
\begin{pmatrix}
-2\sqrt{1-\frac{\Lambda}{\lambda+2\mu}} & \ir \left(2-\frac{\Lambda}{\mu}\right)
\\
-\ir\mu\left(2-\frac{\Lambda}{\mu}\right)  & -2\mu \sqrt{1-\frac{\Lambda}{\mu}}
\end{pmatrix}
\begin{pmatrix}
c_1
\\
c_2
\end{pmatrix}
=
\begin{pmatrix}
0
\\
0
\end{pmatrix}.
\end{equation*}
Therefore, the characteristic equation reads
\begin{equation*}
\chi(\Lambda)=4\sqrt{1-\frac{\Lambda}{\lambda+2\mu}}\sqrt{1-\frac{\Lambda}{\mu}}-\left(2-\frac{\Lambda}{\mu}\right)^2=0\,, \qquad \Lambda\in (0,\mu).
\end{equation*}
We observe that 
\begin{equation}
\label{eq:prooflemEV2DfreeEq2bis}
\chi(\mu\Lambda)=\tilde R_\alpha(\Lambda),
\end{equation}
cf.~\eqref{eq:RayleighEquiv}. But $\tilde R_\alpha(\Lambda)=0$ has a unique solution $\Lambda=\gamma_R^2$ in $(0,1)$, as discussed in Remark~\ref{rem:Rayleigh}. Hence,  \eqref{eq:prooflemEV2DfreeEq2bis} implies \eqref{eq:LambdaR}.

For $\Lambda\in I^{(1)}$ we seek an eigenfunction in the form \eqref{eq:proofEV2DdirEq4}. Substituting \eqref{eq:proofEV2DdirEq4} into \eqref{eq:BC2Dfree}, one sees that the latter can only be satisfied if $c=0$. Therefore, there are no eigenvalues in $I^{(1)}$.

Lastly, by looking at \eqref{eq:EFessspec2D2},  it is easy to see that there are no (square-integrable) eigenfunctions corresponding to $\Lambda\in I^{(2)}$. This concludes the proof.
\end{proof}

Observe that $\Lambda_R<\Lambda_*^{(1)}$, that is, the Rayleigh eigenvalue is located below the continuous spectrum. 

Combining \eqref{eq:argdet2Dfree}, Lemma~\ref{lem:thresholds2Dfree}, and Lemma~\ref{lem:EV2Dfree}, one obtains Lemma~\ref{lem:shiftfree2D}.

Note that in Lemma~\ref{lem:shiftfree2D} there is no distinction between the cases $\lambda\ne 0$ and $\lambda=0$, which \emph{prima facie} may seem at odds with Lemma~\ref{lem:thresholds2Dfree}. However, there is no contradiction, because the different values of $j_*^{(2)}$ in the two cases are compensated exactly by the different values of the jump of \eqref{eq:argdet2Dfree} at the threshold $\Lambda_*^{(2)}$:
\begin{equation*}
\frac1\pi\,\lim_{\epsilon\to 0^+}\left(\arg \det S_{\free,\plane}(\Lambda_*^{(2)}+\epsilon)-\arg \det S_{\free,\plane}(\Lambda_*^{(2)}-\epsilon)\right)=
\begin{cases}
-1 & \text{for }\lambda\ne 0,\\
0 & \text{for }\lambda=0.
\end{cases}
\end{equation*}

\section*{Acknowledgements}
\addcontentsline{toc}{section}{Acknowledgements}

MC was partially supported by a Leverhulme Trust Research Project Grant RPG-2019-240, by a Research Grant (Scheme 4) of the London Mathematical Society, and by a grant of the Heilbronn Institute for Mathematical Research (HIMR) via the UKRI/EPSRC Additional Funding Programme for Mathematical Sciences. 
ML was partially supported by the EPSRC grants  EP/W006898/1 and EP/V051881/1 and by the University of Reading RETF Open Fund.  All the support is gratefully acknowledged.

\begin{appendices} 
\section{On the paper {\cite{Liu}}}
\label{appendix:Liu}

In this appendix, we continue the discussion of the paper \cite{Liu}, which we started in Remark~\ref{rem:Liu}.

Let us begin by observing that, at a qualitative level, the two-term expansion \cite[Theorem 1.1]{Liu} cannot be correct, because the two leading coefficients in the heat trace expansion have the same structure in the way the Lam\'e parameters appear in these coefficients. The boundary conditions mix up longitudinal and transverse waves; hence one expects that contributions from the Lam\'e parameters $\lambda$ and $\mu$ would mix up in a rather complicated way in the second coefficient. Mathematically, the reason for the erroneous expressions comes from the fact that the method of images does not work for the operator of linear elasticity, as we explain below.

For simplicity, in the spirit of Fact~\ref{fact:6}, we work in Euclidean space $\mathbb{R}^2$ endowed with coordinates $(x,y)$, with $\Omega:=\{(x,y)\,| \, y>0\}$ being the upper half-plane and its boundary $\partial\Omega$ being the $x$-axis $\{ y=0\}$.

Let $\tau: (x,y)\mapsto (x, -y)$ be a reflection with respect to the $x$-axis. For a function $u$ (vector- or scalar-valued), consider the involution $Ju:=u\circ\tau$, so that $(Ju)(x,y)=u(x,-y)$. 
It is obvious that $J$ commutes with either the vector or the scalar Laplacian on sufficiently smooth functions: 
\[
(J\circ \Delta)u = (\Delta\circ J)u=\left.(\Delta u)\right|_{\text{evaluated at the point }(x,-y)}.
\]      

Let now $\mathcal{L}$ be the elasticity operator, which acts on vector-valued functions $\mathbf{u}(x,y)=\begin{pmatrix}u_1(x,y)\\u_2(x,y)\end{pmatrix}$ as
\[
\mathcal{L}\mathbf{u}=\begin{pmatrix}-\mu\Delta-(\lambda+\mu)\partial^2_{xx}&-(\lambda+\mu)\partial^2_{xy}\\%
-(\lambda+\mu)\partial^2_{xy}&-\mu\Delta-(\lambda+\mu)\partial^2_{yy}\end{pmatrix}\begin{pmatrix}u_1\\u_2\end{pmatrix}   
=\begin{pmatrix}-(\lambda+2\mu)\partial^2_{xx} u_1-\mu\partial^2_{yy}u_1-(\lambda+\mu)\partial^2_{xy} u_2\\%
-(\lambda+2\mu)\partial^2_{yy} u_2-\mu\partial^2_{xx}u_2-(\lambda+\mu)\partial^2_{xy} u_1\end{pmatrix}.
\]
Then 
\[
(J\circ\mathcal{L})\mathbf{u}=\left.(\mathcal{L} \mathbf{u})\right|_{\text{evaluated at the point }(x,-y)},
\]
but by the chain rule
\[
(\mathcal{L}\circ J)\mathbf{u}=\left.\begin{pmatrix}-(\lambda+2\mu)\partial^2_{xx} u_1-\mu\partial^2_{yy}u_1+(\lambda+\mu)\partial^2_{xy} u_2\\
-(\lambda+2\mu)\partial^2_{yy} u_2-\mu\partial^2_{xx}u_2+(\lambda+\mu)\partial^2_{xy} u_1\end{pmatrix}\right|_{\text{evaluated at the point }(x,-y)}
\]
(i.e., \emph{the signs of the off-diagonal terms involving mixed derivatives  change}), and therefore the commutator $[J, \mathcal{L}]$ does not vanish. Since the elasticity operator  does not commute with reflections, the reflection method (or the method of images) is inapplicable.  

\

The above argument shows that the principal symbol of the Laplacian (or the Laplace--Beltrami operator when working in curved space) is invariant under reflection, whereas the principal symbol of the operator of linear elasticity is not. This is what makes the method of images work for the Laplacian, but not for the operator of linear elasticity. The key difference between the Laplacian (or the Laplace--Beltrami operator) and the operator of linear elasticity is the presence of mixed derivatives in the leading term of the latter.

\

For the reader's convenience, let us spell out the precise points in \cite{Liu} where the main mistakes occur, as a result of the operator $\mathcal{L}$ not being invariant under reflections. Below we use the notation from \cite{Liu}.

The author defines $\mathcal{M}:=\Omega\cup \partial\Omega \cup \Omega^*$ to be the ``double" of $\Omega$, and $\mathcal{P}$ to be the ``double" of the operator of linear elasticity in $\mathcal{M}$. In the simplified setting of this appendix, $\mathcal{M}=\mathbb{R}^2$ and 
\[
\mathcal{P}:=
\begin{cases}
\mathcal{L} & \text{in }\Omega=\{y>0\}\,,
\\
\mathcal{L}+2(\lambda+\mu)
\begin{pmatrix}
0 & \partial_x\\
\partial_x & 0
\end{pmatrix}
\frac{\partial}{\partial y} & \text{in }\Omega^*=\{y<0\}\,.
\end{cases}
\]
Given $\mathbf{u},\mathbf{v}\in C^\infty_c(\mathbb{R}^2)$ (the space of infinitely smooth functions with compact support), by a straightforward integration by parts one obtains
\begin{equation}
\label{P not symmetric}
(\mathcal{P}\mathbf{u},\mathbf{v})-(\mathbf{u},\mathcal{P}\mathbf{v})
=
4(\lambda+\mu)
\int_{\{y=0\}} \left(\partial_x u_1 \, \overline{v_2} - u_2 \,\overline{\partial_x v_1} \right)\,\dr x\,.
\end{equation}
Here $(\cdot\,,\,\cdot)$ denotes the natural $L^2$ inner product and overline denotes complex conjugation. But \eqref{P not symmetric} implies that $\mathcal{P}$ is not symmetric; therefore, it does not give rise to a heat operator.

As a result, the statement
\emph{``$\mathcal{P}$ generates a strongly continuous semigroup $\left(e^{t\mathcal{P}}\right)_{t\ge 0}$ on $L^2(\mathcal{M})$ with integral kernel $K(t,x,y)$"}
in \cite[p.~10169, third line after (1.14)]{Liu} is wrong, and all the analysis based on it breaks down (including \cite[formula (4.3)]{Liu}).

\

In \cite{Liu}, the author states that they borrow their technique from McKean and Singer \cite{McKS}. Indeed, the paragraph preceding formula (4.3) in \cite[p.~10183]{Liu} is taken, almost verbatim, from the beginning of Section~5 of \cite[p.~53]{McKS}. However, McKean and Singer applied the method of images to the Laplacian, for which the ``double" operator is self-adjoint. 

\

Let us conclude this appendix with a brief historical account.
We note that the expression for $\tilde b_\Dir$ was already found\footnote{Not in a mathematically rigorous way, and for specific heat.} in the 1960 paper by M.~Dupuis, R.~Mazo, and L.~Onsager\footnote{The 1968 Nobel Laureate (Chemistry).} \cite{Onsager}. Remarkably, this paper includes the critique of the 1950 paper by E.~W.~Montroll who presented \emph{exactly} Liu's expression \eqref{eq:BLiu} for the second  asymptotic coefficient, modulo some scaling, see \cite[formulae (3)--(5)]{Mon}. Dupuis, Mazo, and Onsager  wrote, we quote: ``\emph{Montroll \dots pointed out in 1950 a defect in the usual counting process of the normal modes of vibration and derived a corresponding correction term for the Debye frequency spectrum, \dots,
proportional to the area of the solid. But it must be clearly realized that he used as a model a parallelopiped with perfectly reflecting faces, and that such boundary conditions are not realistic. It is well known that in the case of a free surface as well as in the case of a clamped surface, one cannot satisfy the boundary conditions by the simple superposition of an incident wave and of a reflected wave of the same kind; one must add a transverse reflected wave if the incident is longitudinal and vice versa. The surface "scrambles" the waves so that one can no longer analyze the vibrations of the solid in terms of pure transverse and pure longitudinal modes.}'' 

The results of Dupuis, Mazo, and Onsager for $d=3$ were reproduced, for both the Dirichlet and the free boundary conditions, as rigorous theorems\footnote{We acknowledge that the level of detail in \cite{SaVa} is somewhat insufficient, as observed in Remark \ref{rem:various}(ii). One should also note that the spectral parameter used in \cite{SaVa} is the square of our spectral parameter $\Lambda$.} in \cite[Section 6.3]{SaVa}, who also extended these results to the planar case $d=2$. 

\section{A two-dimensional example: the disk}\label{sec:ex2d}
Our aim in this (and in the next) appendix to give an experimental verification of the correctness of the second  asymptotic coefficients \eqref{eq:maintheoremDir} and \eqref{eq:maintheoremfree} and to demonstrate the incorrectness of the second  asymptotic coefficient \eqref{eq:BLiu}--\eqref{eq:tildeBLiu}. We work with counting functions rather than with partition functions since the former can, in some cases, be explicitly\footnote{As we will see, such calculations involve solving some transcendental equations, which can be easily done pretty accurately numerically. Still, ``explicitly" should not be taken literally.} calculated for reasonably large values of the parameter, whereas computing the latter would require additional trickery.  

Let $\Omega\subset\mathbb{R}^d$ be the unit disk, equipped with standard polar  coordinates $(r,\phi)$. The equations of elasticity \eqref{eq:Lspec} with Dirichlet boundary conditions in the disk allow the separation of variables\footnote{Like balls in higher dimensions and unlike rectangles, or boxes in higher dimensions --- this was already known to Debye.}. To this end, we, in essence, separate variables in the three-dimensional cylinder $\Omega\times\mathbb{R}$ following \cite[Chapter XIII]{MoFe} and looking for solutions independent of the third coordinate, cf. also \cite[Supplementary materials]{LMS}. More precisely, we take
\begin{equation}\label{eq:udisk}
u(r,\phi)=\grad \psi_1(r,\phi)+\curl\left(\mathbf{z}\,\psi_2(r,\phi)\right),
\end{equation}
where $\mathbf{z}$ is the third coordinate vector. Then it is easily seen that the scalar potentials $\psi_j(r,\phi)$, $j=1,2$, should satisfy the Helmholtz equations
\begin{equation}\label{eq:scHelm}
-\Delta \psi_j=\omega_{j,\Lambda} \psi_j,
\end{equation}
where
\begin{equation}\label{eq:kappas}
\omega_{1,\Lambda}:=\frac{\Lambda}{\lambda+2\mu},\qquad \omega_{2,\Lambda}:=\frac{\Lambda}{\mu}.
\end{equation} 
The general solution of \eqref{eq:scHelm} regular at the origin is well-known,
\begin{equation}\label{eq:Helmsol}
\psi_j(r,\phi)=c_{j,0} J_0\left(\sqrt{\omega_{j,\Lambda}} r\right)+\sum_{k=1}^\infty J_k\left(\sqrt{\omega_{j,\Lambda}} r\right) \left(c_{j,k,+} \mathrm{e}^{\mathrm{i}k\phi}+c_{j,k,-} \mathrm{e}^{-\mathrm{i}k\phi}\right),
\end{equation} 
where the $J_k$ are Bessel functions, and the $c$'s are constants. Substituting \eqref{eq:udisk}, \eqref{eq:kappas}, and \eqref{eq:Helmsol} into the boundary condition $u(1, \phi)=0$ leads, after simplifications, to the \emph{secular equations}
\begin{equation}\label{eq:sec0}
J_1\left(\sqrt{\omega_{1,\Lambda}}\right)J_1\left(\sqrt{\omega_{2,\Lambda}}\right)=0,\qquad\text{if }k=0,
\end{equation}
and
\begin{equation}\label{eq:seck}
\begin{split}
&k \sqrt{\omega _{2,\Lambda }} J_k\left(\sqrt{\omega _{1,\Lambda }}\right)
J_{k+1}\left(\sqrt{\omega _{2,\Lambda }}\right)\\
&+\sqrt{\omega _{1,\Lambda }}J_{k+1}\left(\sqrt{\omega _{1,\Lambda }}\right) \left(k J_k\left(\sqrt{\omega_{2,\Lambda }}\right)
-\sqrt{\omega _{2,\Lambda }} J_{k+1}\left(\sqrt{\omega_{2,\Lambda }}\right)\right)=0,\qquad\text{if }k>0.
\end{split}
\end{equation}
Every solution $\Lambda>0$ of the secular equation \eqref{eq:sec0} is an eigenvalue of multiplicity one of the Dirichlet elasticity operator $\mathcal{L}^\Dir$ on the unit disk, 
and  every solution $\Lambda>0$ of the secular equation \eqref{eq:seck} is an eigenvalue of multiplicity two.

Note that for the case of the disk the branching Hamiltonian billiards associated with the operator of elasticity can be analysed explicitly, and one can check that the two-term asymptotics \eqref{eq:N2term} is valid.

The numerical results are shown in Figure \ref{fig:disk}.

\begin{figure}[htb]
\begin{center}
\includegraphics[width=\textwidth]{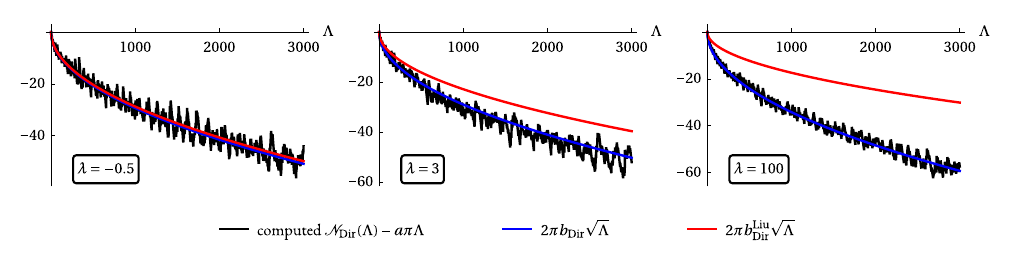}
\caption{The Dirichlet problem for the unit disk. In this case, \eqref{eq:N2term}  may be interpreted as $\mathcal{N}_\Dir(\Lambda)-a\pi\Lambda\approx  2\pi b_\Dir\sqrt{\Lambda}$,
and we compare the numerically computed left-hand sides with the right-hand sides computed using our and Liu's expressions for $b_\Dir$. In all three images we take $\mu=1$.}\label{fig:disk}
\end{center}
\end{figure}

The free boundary problem for the disk is treated in the same manner, the results are shown in Figure \ref{fig:diskfree}.
 
\begin{figure}[htb]
\begin{center}
\includegraphics[width=\textwidth]{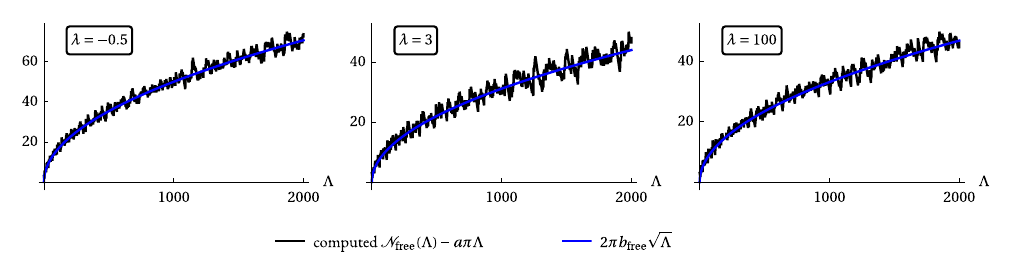}
\caption{The free boundary counting function for the unit disk: in all three figures $\mu=1$.}\label{fig:diskfree}
\end{center}
\end{figure}

\section{Three-dimensional examples: flat cylinders}\label{sec:ex3d}

We consider\footnote{This example goes back to \cite{Onsager}.} $\Omega=\Omega_{3,h}:=\mathbb{T}^2\times[0,h]$, where $\mathbb{T}^2$ is a flat square torus with side $2\pi$ and $h>0$ is the height of the cylinder, so that $\Vol_3(\Omega)=(2\pi)^2 h$ and $\Vol_2(\partial\Omega)=8\pi^2$. We can once again separate variables by first setting
\begin{equation}\label{eq:ucyl3}
u(x^1,x^2,x^3)=\grad \psi_1(x^1,x^2,x^3)+\curl\left(\mathbf{z}\,\psi_2(x^1,x^2,x^3)\right)+\curl\curl\left(\mathbf{z}\,\psi_3(x^1,x^2,x^3)\right), 
\end{equation}
where the $\psi_j$ are scalar potentials, and $\mathbf{z}$ is the coordinate vector in the direction of $x^3$.  Once again, it is easy to see that each potential $\psi_j$ satisfies \eqref{eq:scHelm}, \eqref{eq:kappas}, with
\begin{equation}\label{eq:kappa3}
\omega_{3, \Lambda}:=\omega_{2, \Lambda}.
\end{equation}
The general solutions of \eqref{eq:scHelm} are now
\begin{equation}\label{eq:HelmsolC}
\begin{split}
\psi_j(x^1,x^2,x^3)=\sum_{(k_1, k_2)\in \mathbb{Z}^2} &c_{j, k_1, k_2,+} \exp\left(\mathrm{i}\left(k_1 x^1+k_2 x^2+\sqrt{\omega_{j,\Lambda}-k_1^2-k_2^2}\ x^3\right)\right)\\
+&c_{j, k_1, k_2,-} \exp\left(\mathrm{i}\left(k_1 x^1+k_2 x^2-\sqrt{\omega_{j,\Lambda}-k_1^2-k_2^2}\ x^3\right)\right).
\end{split}
\end{equation}

Substitution of \eqref{eq:ucyl3},  \eqref{eq:kappas}, \eqref{eq:kappa3}, and \eqref{eq:HelmsolC} into the boundary conditions at $x^3=0$ and $x^3=h$ leads to some secular equations\footnote{We omit quite complicated explicit formulae and note only that in this case the numerical solution of these equations is non-trivial, in particular in the free boundary case.} which, as it turns out, depend only on the values of 
\[
K:=k_1^2+k_2^2\in\mathbb{N}\cup\{0\}
\]
rather than on the values of $k_1, k_2$ themselves; we therefore only need to consider the values of $K$ with $\Sigma_2(K)>0$ where
\[
\Sigma_2(K):=\#\left\{(k_1, k_2)\in \mathbb{Z}^2: k_1^2+k_2^2=K\right\}
\]
is the \emph{sum of squares} function. Each solution $\Lambda>0$ of a secular equation corresponding to such a $K$ will be an eigenvalue of multiplicity $\Sigma_2(K)$.  

The results of our computations in the Dirichlet case are collated in Figure \ref{fig:cyl3} and in the free boundary case in Figure \ref{fig:cyl3free}. 

\begin{figure}[htb] 
\begin{center}
\includegraphics[width=\textwidth]{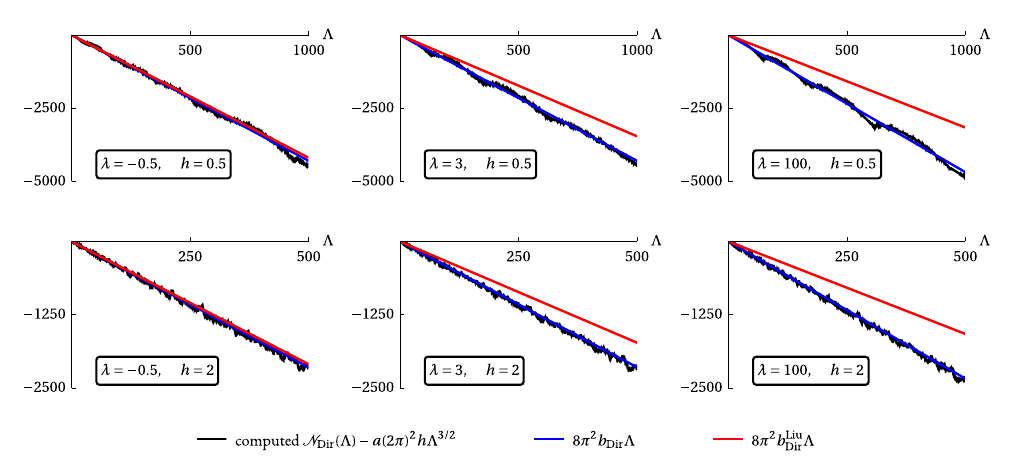}
\caption{The Dirichlet problem for flat cylinders. In all images $\mu=1$.}\label{fig:cyl3}
\end{center}
\end{figure}

\begin{figure}[htb]
\begin{center}
\includegraphics[width=\textwidth]{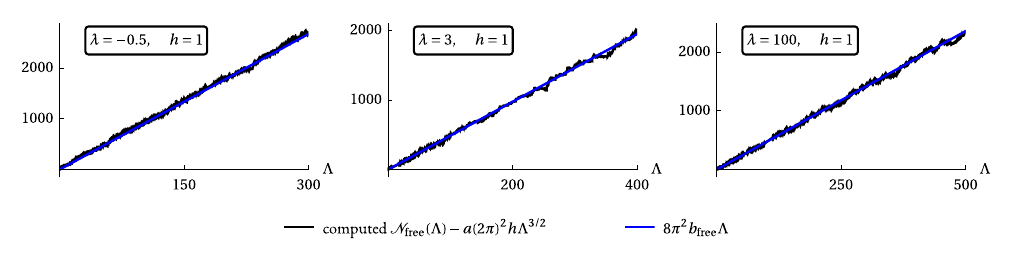}
\caption{The free boundary problem for flat cylinders. In all images $\mu=1$.}\label{fig:cyl3free}
\end{center}
\end{figure}

\section{{Second Weyl coefficients in odd dimensions: proof of Theorem~\ref{thm:formulaeOddDim}}}
\label{appendix:proof}

This appendix is devoted to the proof of Theorem~\ref{thm:formulaeOddDim}.  We will prove \eqref{eq:maintheoremDirOdd} and \eqref{eq:maintheoremfreeOdd} separately, in \S\ref{subsec:dirichletodd} and~\S\ref{subsec:freeodd} respectively, by explicitly evaluating the integrals in the right-hand sides of \eqref{eq:maintheoremDir} and \eqref{eq:maintheoremfree}.

\

In this appendix we denote complex variables by
\begin{equation*}
w=t+\ir s, \qquad t,s\in \mathbb{R}.
\end{equation*}

\subsection{{Dirichlet case: proof of \eqref{eq:maintheoremDirOdd}}}
\label{subsec:dirichletodd}

We begin by observing that, by performing a change of variable $t=\tau^{-2}$,  one obtains
\begin{equation*}
\int \limits_{\sqrt{\alpha}}^1 \tau^{2k-1}\arctan\left(\sqrt{(1-\alpha\tau^{-2})\left(\tau^{-2}-1\right)} \right)\dr\tau
=
\frac12
\int \limits_1^{1/\alpha} \frac{\arctan\left(\sqrt{(1-\alpha t)\left(t-1\right)} \right)}{t^{k+1}} \dr t\,.
\end{equation*}

Therefore,  proving \eqref{eq:maintheoremDirOdd} reduces to establishing the following

\begin{lem}
\label{lem:IkDir}
For $k=1,2,\ldots$ we have
\begin{equation}
\label{eq:lemIkDirEq1}
\frac{2k}{\pi}\int \limits_1^{1/\alpha} \frac{\arctan\left(\sqrt{(1-\alpha t)\left(t-1\right)} \right)}{t^{k+1}} \dr t
=
\left.
\frac{1}{k!}\frac{\dr^k}{\dr t^k}
\left(
\frac{2t-\frac{\alpha+1}{\alpha}}{t-\frac{\alpha+1}{\alpha}}\frac{1}{\sqrt{(1-\alpha t)(1-t)}}
\right)
\right|_{t=0}
-
\left(\frac{\alpha}{\alpha+1}\right)^k\,.
\end{equation}
\end{lem}

\begin{proof}
The task at hand is to evaluate the integral
\begin{equation*}
\mathcal I_k:=\int \limits_1^{1/\alpha} \frac{\arctan\left(\sqrt{(1-\alpha t)\left(t-1\right)} \right)}{t^{k+1}} \dr t
\end{equation*}
for $k\in \mathbb{N}$. Observing that the inverse tangent turns to zero at the endpoints of the interval of integration and integrating by parts, one obtains
\begin{equation*}
\mathcal I_k=\frac1{2k} \int \limits_1^{1/\alpha} \frac1{t^{k+1}}\left(\frac{2t-\frac{\alpha+1}{\alpha}}{t-\frac{\alpha+1}{\alpha}} \frac1{\sqrt{(1-\alpha t)\left(t-1\right)}} \right) \dr t\,.
\end{equation*}
Let
\begin{equation*}
f_k(w):=\frac1{2k}\frac1{w^{k+1}}\left(\frac{2w-\frac{\alpha+1}{\alpha}}{w-\frac{\alpha+1}{\alpha}} \frac1{\sqrt{(1-\alpha w)\left(w-1\right)}} \right), \qquad w\in \mathbb{C}.
\end{equation*}
It is easy to see that the function $f_k$ with branch cut $[1, \alpha^{-1}]$ is holomorphic in a neighbourhood of $[1, \alpha^{-1}]$ and meromorphic in $\mathbb{C}\setminus [1, \alpha^{-1}]$ with poles at 
\begin{equation*}
w=0 \quad (\text{of order }k+1), \qquad w=1+\alpha^{-1} \quad (\text{of order }1)\,.
\end{equation*}
We choose the branch of the square root so that it is positive above the branch cut and negative below the branch cut.  Note that the poles are not on the branch cut.  

Let $\Gamma_\epsilon$ be a negatively oriented (clockwise) dog-bone contour around $[1,\alpha^{-1}]$, see Figure~\ref{fig:contourDir}. Then
\begin{equation}
\label{eq:proofIkDirEq5}
\mathcal{I}_k=\frac12\,\lim_{\epsilon\to 0}\int_{\Gamma_\epsilon} f_k(w) \, dw\,.
\end{equation}

\begin{figure}[htb]
\centering
\includegraphics{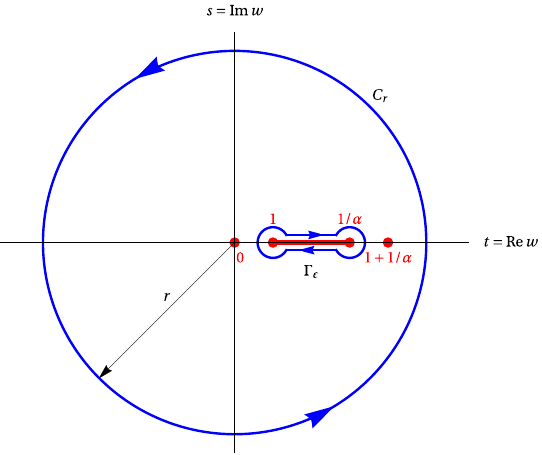}
\caption{The contour of integration for the Dirichlet problem. The two small circles centred at $1$ and $1/\alpha$ have radius $\epsilon$.}\label{fig:contourDir}
\end{figure}

Let $C_r$ be a positively oriented (counterclockwise) circular curve of radius $r$, see Figure~\ref{fig:contourDir}, with $r>1+\alpha^{-1}$. Then by Cauchy's Residue Theorem we have
\begin{equation}
\label{eq:proofIkDirEq6}
\int_{\Gamma_\epsilon\cup C_r} f_k(w)\,\dr w=2\pi \ir \left(\mathrm{Res}(f_k,0)+\mathrm{Res}(f_k,1+\alpha^{-1})\right)\,,
\end{equation}
so that,  combining \eqref{eq:proofIkDirEq5} and \eqref{eq:proofIkDirEq6},  we obtain
\begin{equation}
\label{eq:proofIkDirEq7}
\mathcal{I}_k=\pi \ir \left(\mathrm{Res}(f_k,0)+\mathrm{Res}(f_k,1+\alpha^{-1})+ \mathrm{Res}(f_k,\infty)\right)\,.
\end{equation}

Straightforward calculations give us
\begin{equation}
\label{eq:proofIkDirEq8}
\mathrm{Res}(f_k,0)=\frac1{2k\ir}\frac{1}{k!}\left.\frac{\dr^k}{\dr t^k}\left(\frac{2t-\frac{\alpha+1}{\alpha}}{t-\frac{\alpha+1}{\alpha}}\frac{1}{\sqrt{(1-\alpha t)(1-t)}}  \right)\right|_{t=0}\,,
\end{equation}
\begin{equation*}
\mathrm{Res}(f_k,1+\alpha^{-1})=-\frac1{2k\ir}\,\left(\frac{\alpha}{\alpha+1} \right)^{k}\,,
\end{equation*}
\begin{equation}
\label{eq:proofIkDirEq10}
 \mathrm{Res}(f_k,\infty)=0\,.
\end{equation}
Here we used that $\sqrt{(1-\alpha t)(t-1)}=\ir \sqrt{(1-\alpha t)(1-t)}$ for $t<1$,  and that $\left. \sqrt{(1-\alpha t)(t-1)}\right|_{t=1+\alpha^{-1}}=-\ir$.

Substituting \eqref{eq:proofIkDirEq8}--\eqref{eq:proofIkDirEq10} into \eqref{eq:proofIkDirEq7} we arrive at \eqref{eq:lemIkDirEq1}.
\end{proof}

\subsection{{Free boundary case: proof of \eqref{eq:maintheoremfreeOdd}}}
\label{subsec:freeodd}

As above,  we observe that by performing a change of variable $t=\tau^{-2}$ one obtains
\begin{equation}
\label{eq:chvarfreeodd}
\int\limits_{\sqrt{\alpha}}^1 \tau^{2k-1}\arctan\left(\frac{\left(\tau^{-2}-2\right)^2}{4\sqrt{(1-\alpha\tau^{-2})\left(\tau^{-2}-1\right)}} \right)\dr\tau
=
\frac12\int\limits_{1}^{1/\alpha} \frac{\arctan\left(\frac{\left(t-2\right)^2}{4\sqrt{(1-\alpha t)\left(t-1\right)}} \right)}{t^{k+1}}\dr t\,.
\end{equation}

We have
\begin{lem}
\label{lem:Ikfree}
For $k=1,2,\ldots$ we have
\begin{equation}
\label{eq:lemIkfree}
\begin{split}
&\qquad\frac{4k}{\pi}\int\limits_{1}^{1/\alpha} \frac{\arctan\left(\frac{\left(t-2\right)^2}{4\sqrt{(1-\alpha t)\left(t-1\right)}} \right)}{t^{k+1}}\dr t
\\
&=
-\frac{8}{k!}\left.\frac{\dr^{k}}{\dr t^k}
\left(
\frac{(t-2)(2\alpha t^2+(\alpha-3)t+2(1-\alpha))}{t^3-8t^2+8(3-2\alpha)t+16(\alpha-1)}
 \left(
\frac1{\sqrt{(1-\alpha t)\left(1-t\right)}}
+
\frac4{(t-2)^2}
\right)
\right)
\right|_{t=0}
\\
&+2^{3-k}
+
2(1-\alpha^{k})
-4\gamma_R^{-2k}.
\end{split}
\end{equation}
\end{lem}

\begin{proof}
The task at hand is to evaluate the integral
\begin{equation}
\label{eq:proofIkfreeEq1}
\tilde{\mathcal I}_k:=\int \limits_1^{1/\alpha} \frac{\arctan\left(\frac{\left(t-2\right)^2}{4\sqrt{(1-\alpha t)\left(t-1\right)}} \right)}{t^{k+1}} \dr t
\end{equation}
for $k\in \mathbb{N}$.

In what follows, we assume, for simplicity, that $\alpha\ne\frac{1}{2}$ (this corresponds to $\lambda\ne0$). The case $\alpha=\frac{1}{2}$ can be handled in a similar fashion.

Observing that the inverse tangent tends to $\frac\pi2$ at the endpoints of the interval of integration and integrating by parts, one obtains
\begin{equation}
\tilde{\mathcal I}_k=\frac\pi{2k}(1-\alpha^{k})-\frac2{k} \int \limits_1^{1/\alpha} \frac1{t^{k+1}}\left(\frac{(t-2)(2\alpha t^2+(\alpha-3)t+2(1-\alpha))}{R_\alpha(t)} \frac1{\sqrt{(1-\alpha t)\left(t-1\right)}} \right) \dr t\,,
\end{equation}
where $R_\alpha$ is defined in accordance with \eqref{eq:R(w)}.
Hence,  evaluating \eqref{eq:proofIkfreeEq1} reduces to evaluating 
\begin{equation}
\label{eq:proofIkfreeEq2}
\tilde{\mathcal J}_k:= \int \limits_1^{1/\alpha} \frac1{t^{k+1}}\left(\frac{(t-2)(2\alpha t^2+(\alpha-3)t+2(1-\alpha))}{R_\alpha(t)} \frac1{\sqrt{(1-\alpha t)\left(t-1\right)}} \right) \dr t\,.
\end{equation}
Let
\begin{equation}
\label{eq:proofIkfreeEq3}
\tilde{f}_k(w):=\frac1{w^{k+1}}\left(\frac{(w-2)(2\alpha w^2+(\alpha-3)w+2(1-\alpha))}{R_\alpha(w)} \frac1{\sqrt{(1-\alpha w)\left(w-1\right)}} \right),
\end{equation}
where we choose the branch of the square root in such a way that on the upper side of the branch cut $[0,\alpha^{-1}]$ we have $\sqrt{(1-\alpha w)\left(w-1\right)}>0$.
It is easy to see that the function $\tilde f_k$ is holomorphic in
\begin{equation*}
\mathbb{C}\setminus\left(\{0\}\cup \{w_j,  \ j=1,2,3\} \cup [0,\alpha^{-1}] \right),
\end{equation*}
with poles at
\begin{equation*}
w=0 \quad (\text{of order }k+1), \qquad w=w_j, \ j=1,2,3, \quad (\text{of order }1)\,,
\end{equation*}
where the $w_j$, $j=1,2,3$, are the roots $R_\alpha(w)$, with $0<w_1<1$ --- recall the discussion from Remark~\ref{rem:Rayleigh}.

The nature of the other two roots $w_j$, $j=2,3$, depends on $\alpha$.
Let $\alpha^*\in(0,1)$ be the unique real root of 
\begin{equation*}
\alpha^3-\frac{107}{64}\alpha^2+\frac{31}{32}\alpha-\frac{11}{64}=0.
\end{equation*}
Then we have the three cases \cite{RaBa}:
\begin{enumerate}[(i)]
\label{eq:proofIkfreeEq7}
\item for $0<\alpha<\alpha^*$ there are two complex-conjugate roots $w_2=\overline{w_3}$, $\operatorname{Im}w_2>0$;
\item for $\alpha=\alpha^*$ there are two coinciding real roots $w_2=w_3$;
\item for $\alpha^*<\alpha<1$ there are two distinct real roots $w_2<w_3$.
\end{enumerate}

We will carry out the proof for the case (i); the other two cases are analogous, and lead to the same final result.

\begin{figure}[htb]
\centering
\includegraphics{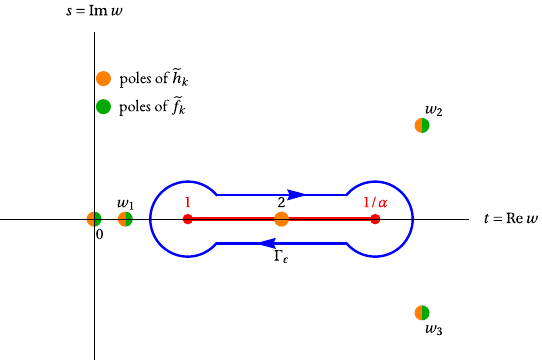}
\caption{The contour of integration for the free boundary problem. }\label{fig:contourfree}
\end{figure}

Let $\Gamma_\epsilon$ be a dog-bone contour as in Figure~\ref{fig:contourfree}.  Then
\begin{equation*}
\tilde{\mathcal J}_k=\frac12\lim_{\epsilon\to 0^+}\int_{\Gamma_\epsilon} \tilde{f}_k(w)\,\dr w\,.
\end{equation*}
Hence, since the function $\tilde f_k$ is regular at infinity,  Cauchy's Residue Theorem gives us
\begin{equation}
\label{eq:proofIkfreeEq9}
\tilde{\mathcal J}_k=\pi\ir\left(\mathrm{Res}(\tilde f_k,0) + \sum_{j=1}^3 \operatorname{Res}(\tilde f_k,w_j)\right)\,.
\end{equation}
From \eqref{eq:proofIkfreeEq3} we immediately obtain
\begin{equation}
\label{eq:proofIkfreeEq10}
\mathrm{Res}(\tilde f_k,0) =-\frac{\ir}{k!}\left.\frac{\dr^{k}}{\dr t^k}
\left(
\frac{(t-2)(2\alpha t^2+(\alpha-3)t+2(1-\alpha))}{R_\alpha(t)} \frac1{\sqrt{(1-\alpha t)\left(1-t\right)}}
\right)
\right|_{t=0}\,.
\end{equation}

In what follows, we evaluate
\begin{equation}
\label{eq:proofIkfreeEq11}
 \operatorname{Res}(\tilde f_k,w_2)+ \operatorname{Res}(\tilde f_k,w_3)-
  \operatorname{Res}(\tilde f_k,w_1)
\end{equation}
and
\begin{equation}
\label{eq:proofIkfreeEq12}
2\operatorname{Res}(\tilde f_k,w_1)\,.
\end{equation}
separately and then add the results together.

Let us start with \eqref{eq:proofIkfreeEq11}.  In view of the equivalence between \eqref{eq:Rayleigh} and \eqref{eq:RayleighEquiv},  and our choice of branch of the square root, it is not hard to check that
\begin{equation}
\label{eq:proofIkfreeEq13}
\left.\sqrt{(1-\alpha w)\left(w-1\right)}\right|_{w=w_1}=\frac{\ir}{4}(w_1-2)^2, 
\end{equation}
\begin{equation*}
\left.\sqrt{(1-\alpha w)\left(w-1\right)}\right|_{w=w_j}=-\frac{\ir}{4}(w_j-2)^2, \qquad j=2,3\,.
\end{equation*}
Hence,  one can recast \eqref{eq:proofIkfreeEq11} as
\begin{equation}
\label{eq:proofIkfreeEq15}
 \operatorname{Res}(\tilde f_k,w_2)+ \operatorname{Res}(\tilde f_k,w_3)-
  \operatorname{Res}(\tilde f_k,w_1)
 =
\sum_{j=1}^3 \operatorname{Res}(\tilde h_k,w_j),
\end{equation}
where
\begin{equation}
\label{eq:proofIkfreeEq16}
\tilde h_k(w):=\frac{4\ir}{w^{k+1}}\left(\frac{2\alpha w^2+(\alpha-3)w+2(1-\alpha)}{(w-2)R_\alpha(w)} \right).
\end{equation}
Now, $\tilde h_k$ is a meromorphic function regular at infinity,  with simple poles at $w=2$ and $w=w_j$, $j=1, 2, 3$, and a pole of order $k+1$ at $w=0$. Therefore, Cauchy's Residue Theorem implies
\begin{equation}
\label{eq:proofIkfreeEq17}
\sum_{j=1}^3 \operatorname{Res}(\tilde h_k,w_j)=-\operatorname{Res}(\tilde h_k,0)-\operatorname{Res}(\tilde h_k,2)\,.
\end{equation}
Combining \eqref{eq:proofIkfreeEq17} and \eqref{eq:proofIkfreeEq15} with account of \eqref{eq:proofIkfreeEq16},  we obtain
\begin{equation}
\begin{split}
\label{eq:proofIkfreeEq18}
&\operatorname{Res}(\tilde f_k,w_2)+ \operatorname{Res}(\tilde f_k,w_3)-
  \operatorname{Res}(\tilde f_k,w_1)
 \\
 &=
 -\frac{4\ir}{k!}\frac{\dr^k}{\dr t^k}\left.\left(\frac{2\alpha t^2+(\alpha-3)t+2(1-\alpha)}{(t-2)R_\alpha(t)} \right)\right|_{t=0}
+ \frac{\ir}{2^{k}}\,.
\end{split}
\end{equation}

Finally, let us deal with \eqref{eq:proofIkfreeEq12}.  Using the fact that $w=w_1$ satisfies \eqref{eq:Rayleigh} and the elementary Vieta's formulae
\begin{equation*}
w_1+w_2+w_3=8, \qquad w_1 w_2 w_3=-16(\alpha-1),
\end{equation*}
one can establish via a straightforward calculation that
\begin{equation}
\label{eq:proofIkfreeEq19}
w_1(w_1-w_2)(w_1-w_3)(w_1-2)=16\left(2\alpha w_1^2+(\alpha-3)w_1+2(1-\alpha)\right)\,.
\end{equation}

Formulae \eqref{eq:proofIkfreeEq13} and \eqref{eq:proofIkfreeEq19} imply
\begin{equation}
\label{eq:proofIkfreeEq20}
2\operatorname{Res}(\tilde f_k,w_1)=-\frac\ir{2 w_1^k}\,.
\end{equation}

Substituting \eqref{eq:proofIkfreeEq10}, \eqref{eq:proofIkfreeEq18} and \eqref{eq:proofIkfreeEq20} into \eqref{eq:proofIkfreeEq9} and then into  \eqref{eq:proofIkfreeEq2} we obtain
\begin{equation}
\begin{split}
\label{eq:proofIkfreeEq22}
\tilde{\mathcal{I}}_k=
&-\frac{2\pi}{k\,k!}\left.\frac{\dr^{k}}{\dr t^k}
\left(
\frac{(t-2)(2\alpha t^2+(\alpha-3)t+2(1-\alpha))}{R_\alpha(t)}
 \left(
\frac1{\sqrt{(1-\alpha t)\left(1-t\right)}}
+
\frac4{(t-2)^2}
\right)
\right)
\right|_{t=0}
\\
&+\frac{2\pi}{k2^{k}}
+
\frac\pi{2k}(1-\alpha^{k})
-\frac\pi{k w_1^k}\,.
\end{split}
\end{equation}
Formula \eqref{eq:lemIkfree} now follows from \eqref{eq:proofIkfreeEq22}
after rescaling and minimal simplifications with account of $w_1=\gamma_R^2$.
\end{proof}

Lemma~\ref{lem:Ikfree} and \eqref{eq:chvarfreeodd} immediately imply \eqref{eq:maintheoremfreeOdd}.

\end{appendices}

\end{document}